\documentclass[12pt]{article}

\usepackage{amsthm}
\usepackage{amsmath}
\usepackage{amssymb}
\usepackage{enumitem}
\usepackage{yfonts}
\usepackage{mathrsfs}
\usepackage{graphicx}
\usepackage{color}

\theoremstyle{definition}
\newtheorem{theorem}{Theorem}[section]
\newtheorem{lemma}[theorem]{Lemma}
\newtheorem{corollary}[theorem]{Corollary}

\newtheorem{definition}[theorem]{Definition}

\newtheorem{remark}[theorem]{Remark}

\newcommand{\mb}[1]{\mbox{\boldmath$#1$}}
\newcommand{\ds}{\,\,}

\newcommand{\bX}{\mbox{\boldmath$\omega$}}
\newcommand{\bH}{\mbox{\boldmath$H$}}
\newcommand{\bK}{\mbox{\boldmath$K$}}
\newcommand{\bA}{\mbox{\boldmath$A$}}
\newcommand{\bB}{\mbox{\boldmath$B$}}
\newcommand{\bC}{\mbox{\boldmath$C$}}
\newcommand{\bD}{\mbox{\boldmath$D$}}
\newcommand{\bV}{\mbox{\boldmath$V$}}
\newcommand{\bI}{\mbox{\boldmath$I$}}
\newcommand{\bomega}{\mbox{\boldmath $\omega$}}

\title{Discrete projective minimal surfaces}

\author{A.\ McCarthy\\ \small
            School of Arts and Sciences\\ \small
            The University of Notre Dame\\ \small
            PO Box 944, Broadway, NSW 2007\\ \small 
            Australia
            \and
           W.K.\ Schief\\ \small
           School of Mathematics and Statistics\\ \small
           The University of New South Wales\\ \small 
           Sydney, NSW 2052\\ \small
          Australia}

\date{15th June 2017}

\begin{document}

\maketitle

\begin{abstract}
We propose a natural discretisation scheme for classical projective minimal surfaces. We follow the classical geometric characterisation and classification of projective minimal surfaces and introduce at each step canonical discrete models of the associated geometric notions and objects. Thus, we introduce discrete analogues of classical Lie quadrics and their envelopes and classify discrete projective minimal surfaces according to the cardinality of the class of envelopes. This leads to discrete versions of Godeaux-Rozet, Demoulin and Tzitz\'eica surfaces. The latter class of surfaces requires the introduction of certain discrete line congruences which may also be employed in the classification of discrete projective minimal surfaces. The classification scheme is based on the notion of discrete surfaces which are in asymptotic correspondence. In this context, we set down a discrete analogue of a classical theorem which states that an envelope (of the Lie quadrics) of a surface is in asymptotic correspondence with the surface if and only if the surface is either projective minimal or a Q surface. Accordingly, we present a geometric definition of discrete Q surfaces and their relatives, namely discrete counterparts of classical semi-Q, complex, doubly Q and doubly complex surfaces.
\end{abstract}

\section{Introduction}

Projective geometry is arguably ``The Queen of All Geometries''. In 1872, this was recognised by Felix Klein whose pioneering and universally accepted Erlangen program revolutionised the way (differential) geometry was approached and treated.  Thus, Klein \cite{Klein1872} proposed that geometries should be classified in terms of groups of transformations acting on a space (of homogeneous coordinates) with projective geometry being associated with the most encompassing group, namely the group of projective collineations. Apart from placing Euclidean and affine geometry in this context, this
approach gives rise to projective models of a diversity of geometries such as  Lie, M\"obius, Laguerre and  Pl\"ucker line geometries.

Projective differential geometry was initiated by Wilczynski \cite{Wilczynski1907,Wilczynski1908,Wilczynski1909} representing the ``American School'' and later formulated in an invariant manner by the ``Italian School'' whose founding members were Fubini and \^{C}ech. \mbox{Cartan} is but one of a long list of distinguished geometers who were involved in the development of this subject with monographs by Fubini \& \^{C}ech \cite{FubiniCech1926,FubiniCech1927}, Lane \cite{Lane1942}, Finikov \cite{Finikov1937} and, most notably, Bol whose first two volumes \cite{Bol1950,Bol1954} consist of more than 700 pages. Recently, there has been a resurgence of interest in projective differential geometry. Here, we mention the monograph {\sl Projective Differential Geometry. Old and New} by Ovsienko and Tabachnikov \cite{OvsienkoTabachnikov2005} and the {\sl Notes on Projective Differential Geometry} by Eastwood \cite{Eastwood2008}. Interestingly, as pointed out in \cite{PrinceCrampin1984}, projective differential geometry also finds application in General Relativity in connection with geodesic conservation laws.  

Projective geometry has proven to play a central role in both discrete differential geometry and discrete integrable systems theory \cite{bobenko2008}. Thus, it turns out that classical incidence theorems of projective geometry such as Desargues', M\"obius' and Pascal's theorems lie at the heart of the integrable structure prevalent in discrete differential geometry. Their algebraic incarnations provide the origin of the integrability of the associated discrete integrable systems such as the master Hirota (dKP), Miwa (dBKP) and dCKP equations (see \cite{Schief2003,KingSchief2006} and references therein). A survey of this important subject may be found in the monograph \cite{bobenko2008}. Moreover, classical projective differential geometry has been shown to constitute a rich source of integrable geometries and associated nonlinear differential equations \cite{FerapontovSchief1999,Ferapontov2000}.

Here, in view of the development of a canonical discrete analogue of projective differential geometry, we are concerned with the important class of projective minimal surfaces \cite{Bol1954}. These have a great variety of geometric and algebraic properties and are therefore custom made for the identification of a general discretisaton procedure which preserves essential geometric and algebraic features. In this connection, it should be pointed out that projective minimal surfaces have been shown to be integrable (see \cite{FerapontovSchief1999,Schief2002} and references therein) in the sense that the underlying projective ``Gauss-Mainardi-Codazzi equations'' are amenable to the powerful techniques of integrable systems theory (soliton theory) \cite{Schief2002}. It turns out that the discretisation scheme proposed here preserves integrability and even though this important aspect will be discussed in a separate publication, we briefly identify and discuss a discrete analogue of the Euler-Lagrange equations for projective minimal surfaces.

Projective minimal surfaces may be characterised and classified both geometrically and algebraically. Here, we mainly focus on geometric  notions and objects with a view to establishing natural discrete counterparts. Even though, by definition, projective minimal surfaces are critical points of the area functional in projective differential geometry, these may also be characterised in terms of Lie quadrics and their envelopes \cite{Bol1954,sasaki2005}. For any (hyperbolic) surface $\Sigma$ in a three-dimensional projective space $\mathbb{P}^3$, there exists a particular two-parameter family (congruence) of quadrics $Q$ which have second-order contact with $\Sigma$. In general, this congruence of quadrics, which are known as Lie quadrics, admit four additional envelopes $\Omega$ which are termed Demoulin transforms of $\Sigma$ \cite{Bol1954,sasaki2005}. If the asymptotic lines on the surface $\Sigma$ are mapped via the congruence of Lie quadrics to the asymptotic lines of at least one envelope $\Omega$ then we say that $\Sigma$ and $\Omega$ are in asymptotic correspondence. In \cite{mccarthy-schief2017}, we have proposed the term PMQ surface for a surface $\Sigma$ which admits at least one envelope $\Omega$ (of the associated Lie quadrics) which is in asymptotic correspondence with $\Sigma$. A classical theorem \cite{sasaki2005,mayer1932} now states that a PMQ surface is either projective minimal (PM) or a Q surface \cite{Bol1954} (or both). A definition of the interesting but restrictive class of Q surfaces is given in the next section.

Projective minimal surfaces may be classified in terms of the number of distinct (additional) envelopes of the Lie quadrics \cite{Bol1954,sasaki2005,mccarthy-schief2017}. If two envelopes are the same then the projective minimal surface $\Sigma$ is of Godeaux-Rozet type. If there exists only one envelope (apart from $\Sigma$ itself) then $\Sigma$ is of Demoulin type. This classification may also be formulated in terms of certain line congruences, which has the advantage that Demoulin surfaces may be further separated into generic Demoulin surfaces and projective transforms of Tzitz\'eica surfaces. The latter have been discussed in great detail not only in the classical context of affine differential geometry but also in connection with the theory of integrable systems and integrability-preserving discretisations (see \cite{Schief2002,bobenkoschief1999} and references therein). On the other hand, Q surfaces are naturally defined in terms of a so-called semi-Q property which gives rise to the isolation of not only Q surfaces but also other classical classes of surfaces, namely complex, doubly Q and doubly complex surfaces \cite{Bol1954}. The semi-Q property is defined in terms of special generators of Lie quadrics which form so-called Demoulin quadrilaterals. This is made precise in the following section.  

In the following, we first briefly review the classical geometric notions and objects, and classes of surfaces mentioned above and then propose, justify and analyse in detail all of their discrete analogues. It turns out that the discrete and classical theories are remarkably close. Moreover, importantly, it may be argued that the discrete theory is more transparent and, thereby, makes the classical theory more accessible. In this connection, it is observed that it is well known that, in many instances, discrete geometries may be generated from continuous geometries by means of iterative application of transformations such as B\"acklund transformations, thereby preserving the essential features of the continuous geometries (see \cite{bobenko2008} and references therein). In \cite{mccarthy-schief2017}, combinatorial and geometric properties of the afore-mentioned classical Demoulin transformation have been investigated in detail and it turns out that the Demoulin transformation applied to classical PMQ surfaces indeed generates discrete PMQ surfaces of the type proposed here.

\section{Projective minimal surfaces}

\subsection{Algebraic classification of projective minimal surfaces}

We are concerned with surfaces $\Sigma$ in a three-dimensional projective space
$\mathbb{P}^3$ represented by
$\mathbf{[\mbox{\boldmath$r$}]}:\mathbb{R}^2\rightarrow\mathbb{P}^3$,
where $(x,y)\in\mathbb{R}^2$ are taken to be  asymptotic coordinates on $\Sigma$. Since we confine ourselves to hyperbolic surfaces, the asymptotic coordinates are real. By definition of asymptotic coordinates, the second derivatives along the coordinate lines are tangent to $\Sigma$ so that the vector of homogeneous coordinates $\mbox{\boldmath$r$}\in\mathbb{R}^4$ satisfies a pair of linear equations
    \begin{equation*}
        \mbox{\boldmath$r$}_{xx}=p\mbox{\boldmath$r$}_y+\pi\mbox{\boldmath$r$}+\sigma\mbox{\boldmath$r$}_x,\quad
        \mbox{\boldmath$r$}_{yy}=q\mbox{\boldmath$r$}_x+\xi\mbox{\boldmath$r$}+\chi\mbox{\boldmath$r$}_y.
    \end{equation*}
Then, it is well known \cite{Bol1954,FerapontovSchief1999,Schief2002} and may readily be verified that particular homogeneous coordinates, known as the Wilczynski lift  \cite{Wilczynski1907,Wilczynski1908,Wilczynski1909}, may be chosen such that the functions $\sigma$ and $\chi$ vanish. Hence, for convenience, one may parametrise the remaining coefficients of the  ``projective Gauss-Weingarten equations'' according to
    \begin{equation*}
        \mbox{\boldmath$r$}_{xx}=p\mbox{\boldmath$r$}_y+\frac{1}{2}(V-p_y)\mbox{\boldmath$r$},\quad
        \mbox{\boldmath$r$}_{yy}=q\mbox{\boldmath$r$}_x+\frac{1}{2}(W-q_x)\mbox{\boldmath$r$}
    \end{equation*}
in terms of functions $p,q,V$ and $W$. The latter are constrained by the ``projective Gauss-Mainardi-Codazzi equations''
    \begin{align}
        \label{GMC1}p_{yyy}-2p_yW-pW_y&=q_{xxx}-2q_xV-qV_x\\
        \label{GMC2}W_x&=2qp_y+pq_y\\
        \label{GMC3}V_y&=2pq_x+qp_x
    \end{align}
which may be derived from the compatibility condition $\mbox{\boldmath$r$}_{xxyy}=\mbox{\boldmath$r$}_{yyxx}$. It is noted that the Wilczynski lift is unique up to the group of transformations
    \begin{equation*}
        x\rightarrow f(x),\quad y\rightarrow g(y),\quad
        \mbox{\boldmath$r$}\rightarrow\sqrt{f'(x)g'(y)}\,\mbox{\boldmath$r$}
    \end{equation*}
with
    \begin{alignat}{4}
        \label{gauget1}p&\rightarrow p\frac{g'(y)}{[f'(x)]^2},\quad&q&\rightarrow q\frac{f'(x)}{[g'(y)]^2}\\[2mm]
        \label{gauget2}V&\rightarrow \frac{V+S(f)}{[f'(x)]^2},\quad&W&\rightarrow\frac{W+S(g)}{[g'(y)]^2},
    \end{alignat}
where $S$ denotes the Schwarzian derivative
    \begin{equation*}
        S(f)=\frac{f'''}{f'}-\frac{3}{2}\left(\frac{f''}{f'}\right)^2.
    \end{equation*}
The quadratic form
    \begin{equation*}
        pq\,dxdy
    \end{equation*}
is a projective invariant and is known as the projective metric. Throughout the paper, we shall assume that $\Sigma$ is not ruled,
i.e., $pq\neq 0$. 

In view of the classification of projective minimal surfaces, it turns out convenient to define functions $\alpha$ and $\beta$ by
    \begin{equation*}
        \alpha=p^2W-pp_{yy}+\frac{p_y^2}{2},\quad
%
        \beta=q^2 V-qq_{xx}+\frac{q_x^2}{2}
    \end{equation*}
so that the Gauss-Mainardi-Codazzi equations
\eqref{GMC1}-\eqref{GMC3} adopt the form
    \begin{align}\label{gmc}
        &\text{\qquad\quad\,}\frac{\alpha_y}{p}=\frac{\beta_x}{q}\\[1mm]
        (\ln p)_{xy}&=pq+\frac{A}{p},\quad
        A_y=-p\left(\frac{\alpha}{p^2}\right)_x\\[2mm]
        (\ln q)_{xy}&=pq+\frac{B}{q},\quad
        B_x=-q\left(\frac{\beta}{q^2}\right)_y.
    \end{align}
This is directly verified by eliminating the functions $A$ and $B$.

    \begin{definition}
        A surface $\Sigma$ in $\mathbb{P}^3$
        is said to be projective minimal if it is critical for the
        area functional
        $\iint pq\,dxdy.$
    \end{definition}
    \begin{theorem}[\cite{thomsen1925}]\label{elcond}
    A surface $\Sigma$ in $\mathbb{P}^3$ is projective minimal if
    and only if
        \begin{equation}\label{eulerlagrange}
            \frac{\alpha_y}{p}=\frac{\beta_x}{q}=0.
        \end{equation}
    \end{theorem}

There exist classical canonical classes of projective minimal surfaces as listed below. Thus, a projective minimal surface is said to be
    \begin{enumerate}[label=(\alph*)]
        \item\label{class1}
            generic if $\alpha\neq 0$ and $\beta\neq 0$.
        \item\label{class2}
            of Godeaux-Rozet type if $\alpha\neq0,\,\beta=0$ or
            $\alpha=0,\,\beta\neq0$.
        \item\label{class3}
            of Demoulin type if $\alpha=\beta=0$. If, in addition,
            $p=q$ then $\Sigma$ is said to be of Tzitz\'{e}ica type.
    \end{enumerate}
It is noted that, using a gauge transformation of the form \eqref{gauget1}, \eqref{gauget2}, one may normalise $\alpha$
and $\beta$ to be one of $-1,\,1$ or $0$. This normalisation corresponds to canonical forms of the integrable system \eqref{gmc}-\eqref{eulerlagrange} underlying projective minimal surfaces \cite{FerapontovSchief1999}.

\subsection{Geometric classification of projective minimal surfaces}
It turns out that the above algebraic classification of projective minimal surfaces admits a natural discrete analogue. This is the subject of a separate publication. Here, our discretisation procedure is of a geometric nature based on the classical geometric classification scheme of projective minimal surfaces which involves Lie quadrics and their envelopes.
     \begin{definition}\label{liedef}
                Let $\mathbf{[\mbox{\boldmath$r$}]}:\mathbb{R}^2\rightarrow\mathbb{P}^3$ be a parametrisation of a surface
                $\Sigma$ in terms of asymptotic coordinates.
                Let $\mbox{\boldmath$p$}=\mbox{\boldmath$r$}(x,y)$ be a point on
                $\Sigma$ and let $\mbox{\boldmath$p$}_{\pm}$ be two
                additional points on the $x$-asymptotic line passing
                through \mbox{\boldmath$p$}, given by $\mbox{\boldmath$p$}_{\pm}=\mbox{\boldmath$r$}(x\pm\epsilon,y)$. Let $l_{\pm}$ and $l$
                be the three lines tangent to the $y$-asymptotic lines at $\mbox{\boldmath$p$}_{\pm}$ and \mbox{\boldmath$p$} respectively.
                These uniquely
                define a quadric $Q_{\epsilon}$ containing them as rectilinear generators. The Lie quadric at $(x,y)$ is then the
                unique quadric defined by
                    \begin{equation*}
                        Q(x,y)=\underset{\epsilon\rightarrow 0}{\lim}\,\,Q_{\epsilon}(x,y).
                    \end{equation*}
    \end{definition}

        It is important to emphasise that the above definition of a Lie quadric may be shown to be symmetric in $x$ and $y$, that is,
        interchanging $x$-asymptotic lines and $y$-asymptotic lines leads to the same Lie quadric $Q$. This is reflected in the explicit representation of the Lie quadric $Q$
given below \mbox{\cite{Bol1954,Ferapontov2000}}. For brevity, in the
following, notationally, we do not distinguish between a Lie quadric in
$\mathbb{R}\mathbb{P}^3$ and its representation in the space of
homogeneous coordinates $\mathbb{R}^4$.

    \begin{theorem}
        The Lie quadric $Q$ (at a point $(x,y)$) admits the parametrisation
            \begin{equation*}
                Q=\textswab{n}+\mu\textswab{r}^1+\nu\textswab{r}^2+\mu\nu\textswab{r},
            \end{equation*}
        where $\mu$ and $\nu$ parametrise the two families of generators of $Q$ and
        $\{\textswab{r},\textswab{r}^1,\textswab{r}^2,\textswab{n}\}$ is the Wilczynski frame given by
    \begin{equation*}
        \begin{split}
            &\textswab{r}=\mbox{\boldmath$r$},\quad
            \textswab{r}^1=\mbox{\boldmath$r$}_x-\frac{q_x}{2q}\mbox{\boldmath$r$},\quad
            \textswab{r}^2=\mbox{\boldmath$r$}_y-\frac{p_y}{2p}\mbox{\boldmath$r$}\\
            &\textswab{n}=\mbox{\boldmath$r$}_{xy}
            -\frac{p_y}{2p}\mbox{\boldmath$r$}_x -\frac{q_x}{2q}\mbox{\boldmath$r$}_y+\left(\frac{p_yq_x}{4pq}
            -\frac{pq}{2}\right)\mbox{\boldmath$r$}.
        \end{split}
    \end{equation*}
    \end{theorem}
It is observed that the lines $(\textswab{r},\,\textswab{r}^1)$ and
$(\textswab{r},\,\textswab{r}^2)$ are tangent to $\Sigma$, while the line
$(\textswab{r},\,\textswab{n})$ is transversal to $\Sigma$ and plays
the role of a projective normal. It is known as the first directrix
of Wilczynski.
    \begin{definition}
        A surface $\Omega$ parametrised by
        $[\bomega]:\mathbb{R}^2\rightarrow\mathbb{P}^3$ is an envelope
        of the two parameter family of Lie quadrics $\{Q(x,y)\}$ associated with a surface $\Sigma$ if $\bomega(x,y)\in Q(x,y)$ such that
        $\Omega$ touches $Q(x,y)$ at $\bomega(x,y)$.
    \end{definition}
We note that, in particular, $\Sigma$ is itself an envelope of
$\{Q\}$. Generically, there exist four additional envelopes as
stated below \cite{Bol1954}.
    \begin{theorem}\label{envel}
    
    If $\alpha,\,\beta\geq0$ then the Lie quadrics $\{Q\}$ possess
        four real additional envelopes
            \begin{align*}
                \bomega_{++}&=\textswab{n}+\hat{\mu}\textswab{r}^1+\hat{\nu}\textswab{r}^2+\hat{\mu}\hat{\nu}\textswab{r}\\
                \bomega_{+-}&=\textswab{n}+\hat{\mu}\textswab{r}^1-\hat{\nu}\textswab{r}^2-\hat{\mu}\hat{\nu}\textswab{r}\\
                \bomega_{-+}&=\textswab{n}-\hat{\mu}\textswab{r}^1+\hat{\nu}\textswab{r}^2-\hat{\mu}\hat{\nu}\textswab{r}\\
                \bomega_{--}&=\textswab{n}-\hat{\mu}\textswab{r}^1-\hat{\nu}\textswab{r}^2+\hat{\mu}\hat{\nu}\textswab{r},
            \end{align*}
        where
            \begin{equation*}
                \hat{\mu}=\sqrt{\frac{\alpha}{2p^2}},\quad\hat{\nu}=\sqrt{\frac{\beta}{2q^2}}.
            \end{equation*}
     These are distinct if $\alpha,\beta \neq 0$.
    \end{theorem}
    \begin{remark}
        The above envelopes are called the Demoulin transforms of
        $\Sigma$. We denote them by $\Sigma_{++}$,
        $\Sigma_{+-}$, $\Sigma_{-+}$ and $\Sigma_{--}$. In general, by construction, these have first-order contact with the Lie quadrics, while, by definition of Lie quadrics, the surface $\Sigma$ has second-order contact. It turns out that there exists a natural discrete analogue of this classical fact.
    \end{remark}
As indicated in the above theorem, the expressions for $\hat{\mu}$
and $\hat{\nu}$ imply that whether $\alpha$ and $\beta$ vanish or
not is related to the number of distinct envelopes. Accordingly, the geometric
interpretation of the algebraic classification
\mbox{\ref{class1}-\ref{class3}} is then that a projective minimal
surface $\Sigma$ is
    \begin{enumerate}[label=(\alph*)]
        \item
            generic if the set of Lie quadrics $\{Q\}$ has four
            distinct additional envelopes.
        \item
            of Godeaux-Rozet type if $\{Q\}$ has two distinct additional
            envelopes.
        \item
            of Demoulin type if $\{Q\}$ has one additional
            envelope.
    \end{enumerate}
    \begin{remark}
        By virtue of the Gauss-Mainardi-Codazzi equation \eqref{gmc},
        Theorem \ref{envel} implies that a surface $\Sigma$ in
        $\mathbb{P}^3$ is necessarily projective minimal if there exist less than four additional distinct envelopes.
        Specifically, if the Lie quadrics of $\Sigma$ have only two additional distinct envelopes
        then $\Sigma$ is of Godeaux-Rozet type. If the Lie quadrics of $\Sigma$ have only one
        additional envelope then $\Sigma$ is of Demoulin type.
    \end{remark}
    \begin{remark}
        For any fixed $(x,y),$ the points
        $\bomega_{++}(x,y),\,\bomega_{+-}(x,y),\bomega_{--}(x,y)$ and
        $\bomega_{-+}(x,y)$ of the Demoulin transforms of $\Sigma$
        may be regarded as the vertices of a quadrilateral (cf.\ Figure \ref{demquad}) 
            \begin{equation*}
                [\bomega_{++}(x,y),\,\bomega_{+-}(x,y),\,\,\bomega_{--}(x,y),\bomega_{-+}(x,y)]
            \end{equation*}
        which is known as the Demoulin quadrilateral \cite{Bol1954}.
        Then, the parametrisation of the envelopes set down in Theorem \ref{envel} shows that the extended edges
        $[\bomega_{++}(x,y),\,\bomega_{+-}(x,y)]$, $[\bomega_{+-}(x,y),\,\bomega_{--}(x,y)]$,
        $[\bomega_{--}(x,y),\,\bomega_{-+}(x,y)]$ and
        $[\bomega_{-+}(x,y),\,\bomega_{++}(x,y)]$
        are generators of the Lie quadric $Q(x,y)$.
        \end{remark}
    \begin{figure}[h]
        \centering
        \includegraphics[scale=1]{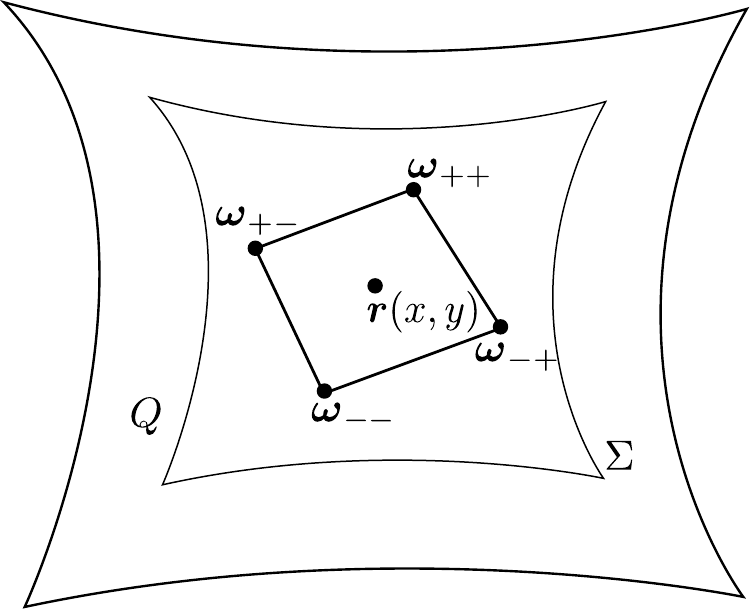}
        \caption{The Demoulin quadrilateral}\label{demquad}
    \end{figure}
Remarkably, it turns out that the Demoulin transformation acts within the class
of projective minimal surfaces and, specifically, within the classes
\ref{class1}-\ref{class3} \cite{sasaki2005,mayer1932}.
    \begin{theorem}\label{grouping1}
        Let $\Sigma$ be a projective minimal surface. Then, each of its Demoulin transforms is projective minimal.
        Moreover, the number $n\in\{1,2,4\}$ of distinct Demoulin transforms
        of $\Sigma$ is preserved by the Demoulin transformation.
        In particular, if $\Sigma$ is of Godeaux-Rozet type
        then each of its Demoulin transforms is of Godeaux-Rozet type. If
        $\Sigma$ is of Demoulin type then its transform is of
        Demoulin type.
    \end{theorem}

The following classical theorem lies at the heart of the geometric definition and analysis of discrete projective minimal surfaces. In order to formulate this theorem, we first note that since any point of a surface $\Sigma$ is mapped to a point of any envelope $\Omega$ via the corresponding Lie quadric, any coordinate system on the surface $\Sigma$ induces a coordinate system on the envelope $\Omega$. Hence, we say that a surface $\Sigma$ and any envelope $\Omega$ are in asymptotic correspondence if the asymptotic lines on $\Sigma$ are mapped to the asymptotic lines on $\Omega$. It turns out that the notion of asymptotic correspondence gives rise to a privileged class of surfaces. Thus, the definition proposed in \cite{mccarthy-schief2017} is adopted.

\begin{definition}
  A surface $\Sigma$ is termed a PMQ surface if it is in asymptotic correspondence with at least one associated envelope $\Omega$.
\end{definition}

The above-mentioned key theorem therefore reads as follows \cite{Bol1954,sasaki2005}.

\begin{theorem}
  The class of PMQ surfaces consists of projective minimal surfaces and Q surfaces.
\end{theorem}

The properties of Q surfaces and their relatives have been discussed in great detail in \cite{Bol1954}. Some of those which are pertinent to the discrete theory developed in the following sections are now briefly mentioned. Thus, in general, for any surface $\Sigma$, any extended edge of the Demoulin quadrilateral displayed in Figure \ref{demquad} such as $[\bomega_{++}(x,y),\bomega_{+-}(x,y)]$ generates a two-parameter family of lines (i.e., a line congruence) by varying the coordinates $(x,y)$ on $\Sigma$. However, this line congruence may degenerate to a one-parameter family of lines, in which case $\Sigma$ is referred to as a semi-Q surface. If this kind of degeneration is present with respect to two connected edges of the Demoulin quadrilateral such as the preceding one and $[\bomega_{++}(x,y),\bomega_{-+}(x,y)]$ then $\Sigma$ is termed a Q surface with respect to the envelope $\Sigma_{++}$. In fact, the latter turns out to be a quadric generated by the two one-parameter families of extended edges. If the line congruences associated with two ``opposite'' edges such as $[\bomega_{++}(x,y),\bomega_{+-}(x,y)]$ and $[\bomega_{--}(x,y),\bomega_{-+}(x,y)]$ degenerate then $\Sigma$ is known as a complex surface. If three or four line congruences degenerate then we refer to $\Sigma$ as doubly Q or doubly complex respectively. However, it is known that doubly Q surfaces are automatically doubly complex. It turns out that the same property holds in the discrete case.

\section{Lattice Lie quadrics}
\subsection{The frame of a discrete asymptotic net}
    \begin{definition}[\cite{bobenko2008}]
        A $\mathbb{Z}^2$ lattice of points in $\mathbb{P}^3$ whose stars are planar is termed a
        discrete asymptotic net.
    \end{definition}
Let $\mb{r}:\mathbb{Z}^2\rightarrow\mathbb{R}^4$ be a homogeneous
coordinate representation of a discrete asymptotic net in
$\mathbb{P}^3$. We usually suppress the arguments in $\mb{r}=\mb{r}(n_1,n_2)$ and denote an increment of $n_k$ by a subscript $k$ and a decrement of $n_k$ by a subscript $\bar{k}$, that is, $\mb{r}_{k}$ and $\mb{r}_{\bar{k}}$ respectively, as illustrated in Figure~\ref{figframe}.
\begin{figure}[h]
    \centering
    \includegraphics[scale=1]{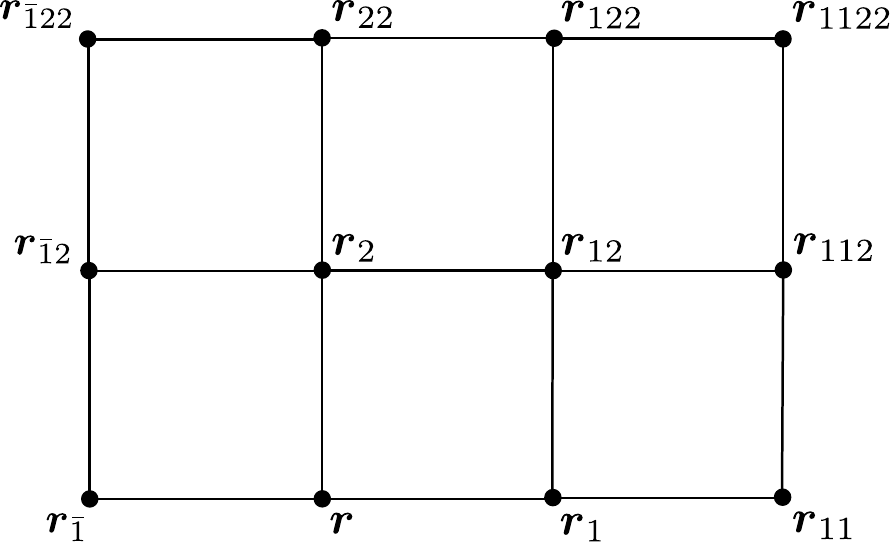}
    \caption{Labeling of a discrete asymptotic net}
    \label{figframe}
\end{figure}
\noindent
We introduce the frame
    \begin{equation}\label{eqfrm1}
        \bV=\left(
                    \begin{array}{l}
                        \mb{r}\\
                        \mb{r}_{1}\\
                        \mb{r}_{2}\\
                        \mb{r}_{12}
                    \end{array}
                \right)
    \end{equation}
so that the planar star condition on the lattice gives
rise to the frame equations
    \begin{align}
        \bV_{1}=&\left(
            \begin{array}{l}
                \mb{r}\\
                \mb{r}_{1}\\
                \mb{r}_{2}\\
                \mb{r}_{12}
            \end{array}
        \right)_{1}
        =\left(
            \begin{array}{cccc}
                0&1&0&0\\
                \alpha^0&\alpha^1&0&\alpha^3\\
                0&0&0&1\\
                0&\beta^1&\beta^2&\beta^3
            \end{array}
        \right)
        \left(
            \begin{array}{l}
                \mb{r}\\
                \mb{r}_{1}\\
                \mb{r}_{2}\\
                \mb{r}_{12}
            \end{array}
        \right)
        =L\bV \label{eqframe1}\\[2mm]
        \bV_{2}=&\left(
            \begin{array}{l}
                \mb{r}\\
                \mb{r}_{1}\\
                \mb{r}_{2}\\
                \mb{r}_{12}
            \end{array}
        \right)_{2}
        =\left(
            \begin{array}{cccc}
                0&0&1&0\\
                0&0&0&1\\
                \gamma^0&0&\gamma^2&\gamma^3\\
                0&\delta^1&\delta^2&\delta^3
            \end{array}
        \right)
        \left(
            \begin{array}{l}
                \mb{r}\\
                \mb{r}_{1}\\
                \mb{r}_{2}\\
                \mb{r}_{12}
            \end{array}
        \right)
        =M\bV, \label{eqframe2}
    \end{align}
where we note that in order to exclude the degenerate case of
three points connected by two edges being collinear, we demand that
only $\alpha^1,\beta^3,\gamma^2$ and $\delta^3$ be allowed
to be zero. $L$ and $M$ are not arbitrary but are constrained by the
compatibility condition $\bV_{12}=\bV_{21}$, that is
    \begin{equation}\label{eqcomp}
        M_{1}L=L_{2}M
    \end{equation}
    which equates to
    \begin{equation}\label{eqcomp1}
        \begin{split}
            &\left(
            \begin{array}{cccc}
                0&0&0&1\\[1mm]
                0&\beta^1&\beta^2&\beta^3\\[1mm]
                0&\beta^1\gamma^3_{1}+\gamma^0_{1}&\gamma^3_{1}\beta^2&\beta^3\gamma^3_{1}+\gamma^2_{1}\\[2mm]
                \delta^1_{1}\alpha^0&\alpha^1\delta^1_{1}+\beta^1\delta^3_{1}&\delta^3_{1}\beta^2&\alpha^3\delta^1_{1}+\beta^3\delta^3_{1}+\delta^2_{1}
            \end{array}
            \right)\\[2mm]
            =
            &\left(
            \begin{array}{cccc}
                0&0&0&1\\[1mm]
                0&\alpha^3_{2}\delta^1&\delta^2\alpha^3_{2}+\alpha^0_{2}&\delta^3\alpha^3_{2}+\alpha^1_{2}\\[2mm]
                0&\delta^1&\delta^2&\delta^3\\[1mm]
                \beta^2_{2}\gamma^0&\beta^3_{2}\delta^1&\gamma^2\beta^2_{2}+\delta^2\beta^3_{2}&\gamma^3\beta^2_{2}+\delta^3\beta^3_{2}+\beta^1_{2}
            \end{array}
            \right).
        \end{split}
    \end{equation}
These may be regarded as the ``Gauss-Mainardi-Codazzi equations'' in
discrete projective differential geometry.

\subsection{Lattice Lie quadrics}

For each quadrilateral of a discrete asymptotic net, there exists a
one-para\-meter family of quadrics passing through the edges of that
quadrilateral. Analogously to the continuous Lie quadric for
projective minimal and Q surfaces, these ``lattice quadrics'' play a central role in
defining discrete Demoulin transforms.
\begin{figure}[h]
    \centering
     \includegraphics[scale=1]{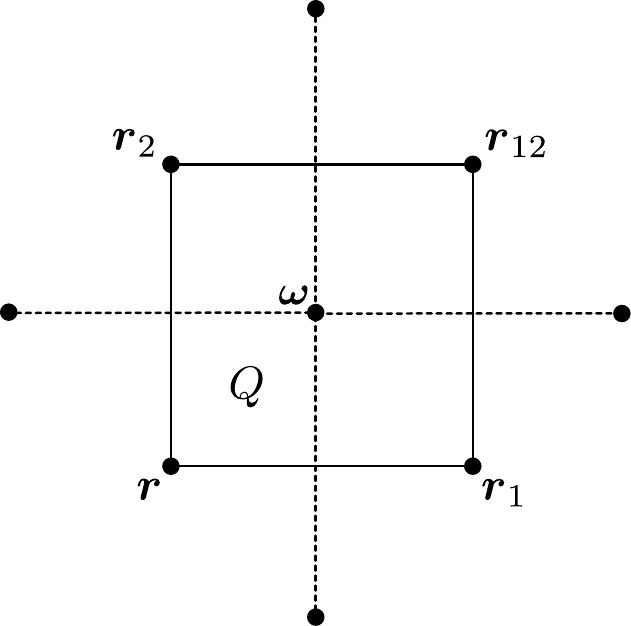}
    \caption{Combinatorial picture of a quadrilateral of a discrete asymptotic net and the star of an envelope touching an associated lattice Lie quadric}
    \label{touch}
\end{figure}
    \begin{definition}
       \quad
        \begin{enumerate}[label=(\roman*)]
            \item
            A quadric associated with a quadrilateral of a discrete asymptotic net
            is a quadric that passes through the edges of the
            quadrilateral. For any given quadrilateral, there
            exists a one parameter family of such quadrics.
            \item
            Let $Q$ and $\hat{Q}$ be two quadrics associated with neighbouring
            quadrilaterals of a discrete asymptotic net. $Q$ and $\hat{Q}$ are
            said to satisfy the $\mathcal{C}^1$ condition if their
            tangent planes coincide at each point of the common edge.
            \item
            A set of quadrics $\{Q\}$ associated with the quadrilaterals of a discrete asymptotic
            net is termed a set of lattice Lie quadrics if the $\mathcal{C}^1$
            condition holds for all neighbouring pairs of quadrics.
            \item
            A discrete envelope of a set of lattice Lie quadrics
            $\{Q\}$ is a dual lattice of $\mathbb{Z}^2$ combinatorics such that each star
            touches the corresponding lattice Lie quadric (cf.\ Figure \ref{touch}).
            \item
            A discrete PMQ surface is a discrete asymptotic net which
            admits a discrete envelope of an associated set of lattice Lie quadrics.
        \end{enumerate}
    \end{definition}
    We parametrise a quadric passing through the edges of the quadrilateral $[\mb{r}\,\,\mb{r}_{1}\,\,\mb{r}_{2}\,\,\mb{r}_{12}]$ by
            \begin{equation}\label{eqqparam}
                Q(s,t)=p\mb{r}_{12}+s \mb{r}_{1}+t \mb{r}_{2}+s t
                \mb{r},
            \end{equation}
        where $p$ is fixed and labels $Q$ among
        the one-parameter family of quadrics associated with the
        quadrilateral, and $s$ and $t$ are the parameters of $Q$ which
        parametrise its generators.

\begin{remark}
  The four (extended) edges of any quadrilateral of a discrete asymptotic net lying on the corresponding lattice Lie quadric may be regarded as a discrete analogue of second-order contact, while any planar star of a discrete envelope touching the corresponding lattice Lie quadric constitutes a discrete analogue of first-order contact.
\end{remark}

In order to establish the existence of lattice Lie quadrics, we first recall a key theorem established in \cite{huhnen-venedey2014} and prove it algebraically.

    \begin{theorem}\label{propc1}
        Given a fixed quadric $Q$, the $\mathcal{C}^1$ condition
        uniquely determines a quadric $Q_{1}$ associated with the neighbouring
        quadrilateral.
    \end{theorem}
    \begin{proof}
        We parametrise $Q$ by \eqref{eqqparam} and $Q_{1}$ by
            \begin{equation}\label{eqq1param}
                Q_{1}(s_{1},t_{1})=p_{1}\mb{r}_{112}+s_{1}\mb{r}_{11}+t_{1}\mb{r}_{12}+s_{1}t_{1}\mb{r}_{1}
            \end{equation}
        so that, by virtue of \eqref{eqframe1}, \eqref{eqframe2},
            \begin{equation}\label{eqq11param}
                \begin{split}
                    Q_{1}=(\beta^3 p_{1}&+\alpha^3 s_{1}+t_{1})\mb{r}_{12}\\&+(\beta^1 p_{1}+\alpha^1 s_{1}+s_{1}
                    t_{1})\mb{r}_{1}+\beta^2 p_{1}\mb{r}_{2}+\alpha^0 s_{1}\mb{r}.
                \end{split}
            \end{equation}
        A point $\mb{X}$ on the edge common to $Q$ and $Q_{1}$ is
        parametrised by $ps_{1}=s$ and $t=0,\,t_{1}=\infty$. The $\mathcal{C}^1$ condition
        is then
            \begin{equation*}
                \left|
                \mb{X},\,\,\mb{r}_{12},\,\,\left.\frac{\partial}{\partial
                t}Q\right|_{\mb{X}},\,\,\left.\frac{\partial}{\partial
                \hat{t}_{1}}\hat{Q}_{1}\right|_{\mb{X}}\right|=0,
            \end{equation*}
        where,  to obtain a finite tangent vector, we have scaled $Q_{1}$ according to $\hat{Q}_{1}=\hat{t}_{1}Q_{1}$ with $\hat{t}_{1}=1/t_{1}$. This
        yields
            \begin{equation}\label{eqc11}
                \left|s_{1}\mb{r}_{1},\,\,\mb{r}_{12},\,\,\mb{r}_{2}+ps_{1}\mb{r},\,\,p_{1}\mb{r}_{112}+s_{1}\mb{r}_{11}\right|=0.
            \end{equation}
        On use of the system \eqref{eqframe1}, \eqref{eqframe2}, this is shown to determine $p_1$ according to
            \begin{equation}\label{eqp11}
                p_{1}=\frac{\alpha^0}{\beta^2 p}.
            \end{equation}
    \end{proof}
    \begin{remark}
        We note that Theorem \ref{propc1} makes use of the fact
        that $Q$ and $Q_{1}$ share tangents plane at the
        points $\mb{r}_{1}$ and $\mb{r}_{12}$. Moreover,
        since \eqref{eqp11} is independent of $s$, it follows that
        the tangent planes of $Q$ and $Q_{1}$ coincide everywhere
        along the common edge, confirming the following known
        result.
    \end{remark}
        \begin{theorem}[\cite{huhnen-venedey2014}]\label{thm3pts}
            Let $Q$ and $Q_{1}$ be quadrics associated with two neighbouring quadrilaterals of a
            discrete asymptotic net. If the tangent planes of $Q$
            and $Q_{1}$ coincide at a point on their common
            edge other than any of the vertices connected by the common edge then the tangent planes to both quadrics coincide everywhere along the common
            edge.
        \end{theorem}
    \begin{remark}
        Let $Q_{2}$ be the quadric passing through the edges of the
        quadrilateral
        $[\mb{r}_{2},\,\mb{r}_{12},\,\mb{r}_{22},\,\mb{r}_{122}]$
        which satisfies the $\mathcal{C}^1$ condition with $Q$. Let
        $Q_{12}$ and $Q_{21}$ be the quadrics passing through the edges of the
        quadrilateral
        $[\mb{r}_{12},\,\mb{r}_{112},\,\mb{r}_{122},\,\mb{r}_{1122}]$
        which satisfies the $\mathcal{C}^1$ condition with $Q_{1}$ and $Q_{2}$
        respectively (cf.\ Figure \ref{quadsonpatch}). Then, expressions similar to
        \eqref{eqp11} show that the parameters $p_{12}$ and $p_{21}$ associated with $Q_{12}$ and $Q_{21}$ respectively coincide.
        Thus, $Q_{12}=Q_{21}$ and, hence, given one fixed quadric associated
        with one quadrilateral of a discrete asymptotic net, the
        $\mathcal{C}^1$ condition uniquely determines all other
        quadrics on the lattice \cite{huhnen-venedey2014}. The implication of this is summarised below.
    \end{remark}
    \begin{figure}[h]
        \centering
        \includegraphics[scale=1]{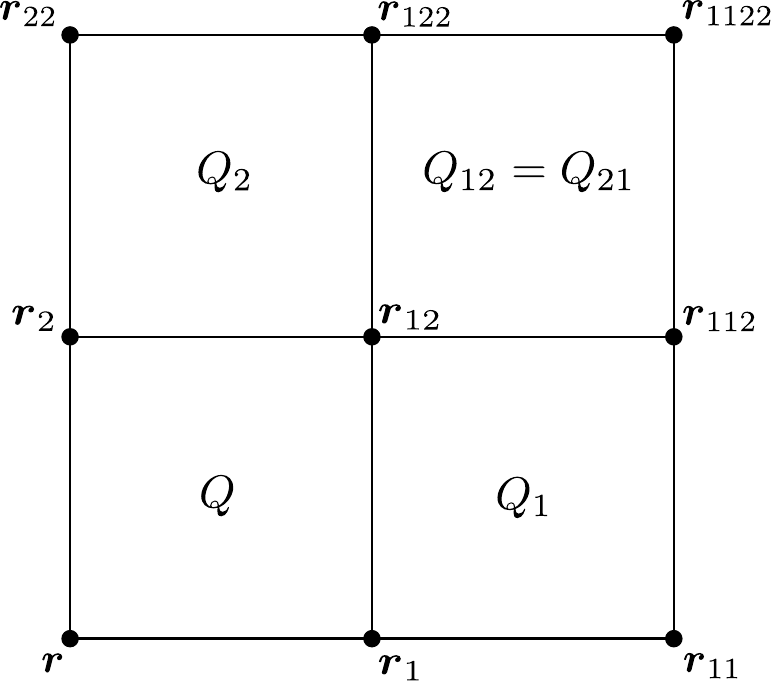}
        \caption{Quadrics on a patch of a discrete asymptotic net related by the $\mathcal{C}^1$ condition}
        \label{quadsonpatch}
    \end{figure}

\begin{theorem}
  A set of lattice Lie quadrics is uniquely determined by a quadric associated with one quadrilateral of a discrete asymptotic net.
\end{theorem}

\subsection{Generators shared by lattice Lie quadrics}

In view of classifying discrete PMQ surfaces and their envelopes, we 
now investigate the properties of quadrics associated with the
quadrilaterals of a discrete asymptotic net.
    \begin{theorem}\label{thmsharegen1}
        Let $Q$ and $Q_{1}$ be two neighbouring quadrics on a
        discrete asymptotic net which satisfy the $\mathcal{C}^1$
        condition. Then, $Q$ and $Q_{1}$ share two (possibly
        complex or coinciding) generators (transversal to the common edge). 
       Conversely, if there exists a generator common to $Q$ and $Q_{1}$ then the
        $\mathcal{C}^1$ condition is satisfied.
    \end{theorem}
    \begin{proof}
        At points common to $Q$ and $Q_{1}$, we have  $Q(s,t)\sim Q_{1}(s_1,t_1)$, which
        yields the equations
            \begin{align}
                \frac{s}{p}&=\frac{\beta^1 p_{1}+\alpha^1
                s_{1}+s_{1} t_{1}}{\beta^3 p_{1}+\alpha^3
                s_{1}+t_{1}}\label{eqsgen1}\\
                \frac{t}{p}&=\frac{\beta^2 p_{1}}{\beta^3 p_{1}+\alpha^3
                s_{1}+t_{1}}\label{eqsgen2}\\
                \frac{st}{p}&=\frac{\alpha^0 s_{1}}{\beta^3 p_{1}+\alpha^3
                s_{1}+t_{1}}.\label{eqsgen3}
            \end{align}
        One solution to these equations is $t=0,\,t_{1}=\infty$ and $ps_{1}=s$, which parametrises the common
        edge. Hence, if $Q$ and $Q_{1}$ have a common
        generator then, necessarily, the corresponding parameters $s$ and $s_{1}$ are related by $ps_{1}=s$. As a result, relations \eqref{eqsgen2} and \eqref{eqsgen3}
        imply that $p_{1}=\alpha^0/(\beta^2p)$, which is precisely the $\mathcal{C}^1$ condition
        \eqref{eqp11}. This is geometrically evident since at the point of intersection of the common edge and generator, the generator and edge span the coinciding tangent
        planes
        of $Q$ and $Q_{1}$ so that, as a result of Theorem \ref{thm3pts}, the tangent
        planes to $Q$ and $Q_{1}$ coincide along their common
        edge.
        Hence, if $Q$ and $Q_{1}$ have a shared generator then they satisfy the $\mathcal{C}^1$ condition. 

Conversely, since any common generator is ``labelled by $ps_{1}=s$'', that is, the labels $s$ and $s_1$ of a generator common to $Q$ and $Q_1$ are related by $ps_1=s$, the $\mathcal{C}^1$ condition \eqref{eqp11} implies that \eqref{eqsgen2} and \eqref{eqsgen3} coincide, thereby relating the parameters $t$ and $t_1$ but leaving one of them arbitrary, and the remaining relation \eqref{eqsgen1} reduces to the condition

            \begin{equation}\label{eqsharegen}
                \alpha^3\beta^2(s)^2+s(\alpha^0\beta^3-\alpha^1\beta^2
                p)-\alpha^0\beta^1p=0.
            \end{equation}
            The discriminant of the latter is
                    \begin{equation}
                        D^1=(\alpha^0\beta^3-\alpha^1\beta^2
                        p)^2+4\alpha^0\alpha^3\beta^1\beta^2 p
                    \end{equation}
and determines whether the roots of \eqref{eqsharegen} are real, complex conjugates or coinciding. Accordingly, the proof is complete and the following corollary holds.
\end{proof}

 \begin{corollary}\label{corsharegen1}
        Let $Q$ and $Q_{1}$ be two neighbouring quadrics on a
        discrete asymptotic net which satisfy the $\mathcal{C}^1$
        condition and 
    \begin{equation}\label{eqdisc}
                        D^1=(\alpha^0\beta^3-\alpha^1\beta^2
                        p)^2+4\alpha^0\alpha^3\beta^1\beta^2 p
     \end{equation}
be the associated discriminant. If
                    \begin{enumerate}[label=(\roman*)]
                        \item
                        $D^1>0$ then  $Q$ and $Q_{1}$ share two distinct
                        real generators.
                        \item
                        $D^1<0$ then $Q$ and $Q_{1}$ share two
                        distinct complex conjugate generators.
                        \item
                        $D^1=0$ then $Q$ and $Q_{1}$ share two coinciding real 
                        generators.
                    \end{enumerate}
\end{corollary}
\smallskip

    \begin{remark}
        Similarly, in the $n_2$ direction, $Q$ and $Q_2$ share two
        generators labelled by $pt_2=t$ and
            \begin{equation}\label{n2constraint}
                \gamma^3\delta^1 (t)^2+t(\gamma^0\delta^3-\gamma^2\delta^1
                p)-\gamma^0\delta^2p=0
            \end{equation}
        with discriminant
            \begin{equation}\label{eqd2}
             D^2=(\gamma^0\delta^3 -\gamma^2\delta^1
                p)^2+4\gamma^0\gamma^3\delta^1\delta^2
                p.
            \end{equation}
    \end{remark}
    In the case $D^1=0$ (or $D^2=0$), we also have the following useful property.
    \begin{corollary}\label{cortouch}
        Let $Q$ and $Q_{1}$ be two neighbouring quadrics on a
        discrete asymptotic net which satisfy the $\mathcal{C}^1$
        condition and assume that $D^1=0$. Then, $Q$ and $Q_{1}$
        touch along their shared generator.
    \end{corollary}
    \begin{proof}
        Since $D^1=0$, the proof of Theorem \ref{thmsharegen1} implies that the shared
        generator is labelled by $ps_{1}=s$, where
            \begin{equation*}
                s=\displaystyle\frac{\alpha^1\beta^2
                p-\alpha^0\beta^3}{2\alpha^3\beta^2}.
            \end{equation*}
        Let $\mb{X}=Q_{1}(s_{1},t_{1})$ be a fixed point on the shared
        generator which does not also lie on the common edge so that
        \eqref{eqsgen2} leads to
            \begin{equation*}
                t=\displaystyle\frac{2\alpha^0\beta^2
                p}{\alpha^1\beta^2 p+2\beta^2 p
                t_{1}+\alpha^0\beta^3}.
            \end{equation*}
    Along a common generator, we then have, trivially,
        \begin{equation*}
            \left|\mb{X},\,\,\left.\frac{\partial}{\partial s}
                Q\right|_{\mb{X}},\,\,\left.\frac{\partial}{\partial t}
                Q\right|_{\mb{X}},\,\,\left.\frac{\partial}{\partial
                t_{1}}Q_{1}\right|_{\mb{X}}\right|=0
        \end{equation*}
    and it is easy to verify that
            \begin{align*}
                \left|\mb{X},\,\,\left.\frac{\partial}{\partial s}
                Q\right|_{\mb{X}},\,\,\left.\frac{\partial}{\partial t}
                Q\right|_{\mb{X}},\,\,\left.\frac{\partial}{\partial s_{1}} Q_{1}\right|_{\mb{X}}\right|=0.
            \end{align*}
    Hence, the tangent planes of $Q$ and $Q_{1}$ coincide at
    $\mb{X}$.
    \end{proof}

\section{Discrete projective minimal surfaces}

We now intend to establish under what circumstances envelopes of sets of lattice Lie quadrics exist.

\subsection{The tangency condition}

    \begin{definition}
        Let $Q$ and $Q_{1}$ be neighbouring quadrics on a discrete
        asymptotic net. Then, $\bX\in Q$ and $\bX_{1}\in Q_{1}$ are said
        to
        satisfy the tangency condition if the line joining
        $\bX$ and $\bX_{1}$ is tangent to both quadrics at the respective
        points.
    \end{definition}
    \begin{theorem}\label{thmuniqpt}
        Let $Q$ and $Q_{1}$ be two neighbouring quadrics on a
        $3\times 2$ patch of a discrete asymptotic net that satisfy the $\mathcal{C}^1$
        condition. Let $\bX\in Q$ be a generic point. Then, there exists a unique point $\bX_{1}\in Q_{1}$ such
                that $\bX$ and $\bX_{1}$ satisfy the tangency
                condition (cf.\ Figure \ref{tangency}).
    \end{theorem}
    \begin{figure}[h]
        \centering
        \includegraphics[scale=1]{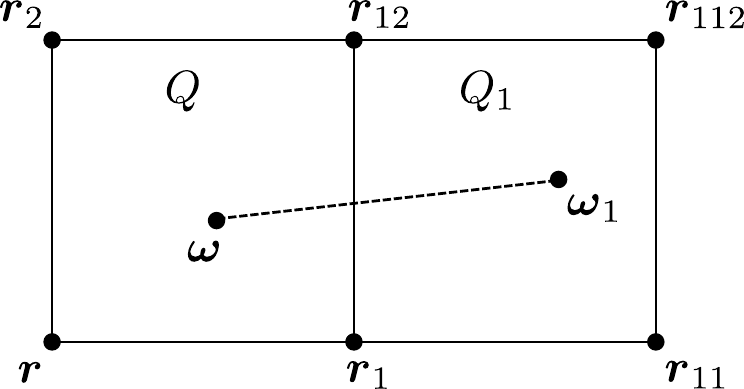}
        \caption{The tangency condition}
        \label{tangency}
    \end{figure}
    \begin{proof}
        Set $\bX=Q(s,t)$, $\bX_{1}=Q_{1}(s_{1},t_{1})$.
        The tangency condition is then encapsulated in the pair
            \begin{align}
                    \left|
                        \bX,\ds \bX_{1},\ds\left.\frac{\partial}{\partial
                        s}Q\right|_{\bX},\ds\left.\frac{\partial}{\partial t}
                        Q\right|_{\bX}
                    \right|&=0\label{eqtang1} \\[2mm]
                    \left|
                        \bX,\ds \bX_{1},\ds\left.\frac{\partial}{\partial
                        s_1}Q_{1}\right|_{\bX_{1}},\ds\left.\frac{\partial}{\partial t_1}
                        Q_{1}\right|_{\bX_{1}}
                    \right|&=0,\label{eqtang2}
            \end{align}
        which yields
            \begin{align}
                (ps_{1}-s)(\alpha^1\beta^2 p t+2\beta^2 p t t_{1}-2\alpha^0 \beta^2 p+\alpha^0\beta^3 t)&=0\label{eqtang3}\\[1mm]
                \left(p
                s_{1}-\displaystyle\frac{\alpha^1\beta^2p-\alpha^0\beta^3}{2\alpha^3\beta^2}\right)\left(s-\displaystyle\frac{\alpha^1\beta^2
                p-\alpha^0\beta^3}{2\alpha^3\beta^2}\right)-\displaystyle\frac{1}{(2\alpha^3\beta^2)^2}D^1&=0\label{eqtang4}.
            \end{align}
        If $ps_{1}=s$ then equation \eqref{eqtang4} reduces to
        equation \eqref{eqsharegen} and, hence, $\bX$ and $\bX_{1}$ lie on
        a common generator so that $\bX$ is non-generic.
        If $ps_{1}\neq s$ then
            \begin{align}
                s_{1}=\displaystyle\frac{\alpha^1\beta^2 p s+2\alpha^0\beta^1 p-\alpha^0\beta^3s}{p(2\alpha^3\beta^2 s-\alpha^1\beta^2p+\alpha^0\beta^3)
                }&=:\mathcal{S}^1(s)\label{eqs2}\\[1mm]
                t_{1}=\displaystyle\frac{2\alpha^0\beta^2 p-\alpha^1\beta^2pt-\alpha^0\beta^3 t}{2\beta^2 p t}&=:\mathcal{T}^1(t),\label{eqt2}
            \end{align}
    where genericity implies that $2\alpha^3\beta^2-\alpha^1\beta^2p
    s+\alpha^0\beta^3\neq0$.
    \end{proof}
    \begin{remark}\label{remarsgen}
            Applying the tangency condition to a generic point $\bX$ in the $n_2$
            direction generates the pair of equations
                \begin{align}
                        (p t_2-t)(\gamma^2\delta^1 p s+2 \delta^1 p
                        s s_2-2\gamma^0\delta^1 p+\gamma^0\delta^3
                        s)&=0\label{eqttanga}\\[1mm]
                        \left(p
                        t_2-\displaystyle\frac{\gamma^2\delta^1
                        p-\gamma^0\delta^3}{2\gamma^3\delta^1}\right)\left(t-\displaystyle\frac{\gamma^2\delta^1
                        p-\gamma^0\delta^3}{2\gamma^3\delta^1}\right)-\displaystyle\frac{1}{(2\gamma^3\delta^1)^2}D^2&=0.\label{eqttangb}
                \end{align}
            If $pt_2=t$ then equation \eqref{eqttangb} gives rise to equation \eqref{n2constraint} and labels the shared
            generators
            in the $n_2$ direction.
            In the generic case, \eqref{eqttanga} and \eqref{eqttangb} yield a point $\bX_2=Q_2(s_2,t_2)$ with
            \begin{align}
                s_{2}=\displaystyle\frac{2\gamma^0\delta^1p
                -\gamma^2\delta^1 ps-\gamma^0\delta^3 s}{2\delta^1 p
                s}&=:\mathcal{S}^2(s)\label{eqs22}\\[1mm]
                t_{2}=\displaystyle\frac{\gamma^2\delta^1p
                t+2\gamma^0\delta^2 p-\gamma^0\delta^3
                t}{p(2\gamma^3\delta^1t-\gamma^2\delta^1 p
                +\gamma^0\delta^3)}&=:\mathcal{T}^2(t)\label{eqt22}.
            \end{align}
        Thus, in the generic case, the tangency condition induces 4
        maps taking $(s,t)$ to $(s_1,t_1)$ and
        $(s_2,t_2)$ which we have denoted by $\mathcal{S}^1,\,\mathcal{T}^1$ and $\mathcal{S}^2,\,\mathcal{T}^2$ respectively.

        We also note that if $ps_{1}\neq s$, and \eqref{eqtang4}
        cannot be solved for either $s$ or $s_{1}$ then,
        necessarily,
            \begin{equation}\label{eqgrd}
                D^1=(\alpha^0\beta^3-\alpha^1\beta^2
                p)^2+4\alpha^0\alpha^3\beta^1\beta^2
                p=0
            \end{equation}
        and \eqref{eqtang4} becomes
            \begin{equation}\label{eqgr1}
                \left(ps_{1}-\displaystyle\frac{\alpha^1\beta^2
                p-\alpha^0\beta^3}{2\alpha^3\beta^2}\right)
                \left(s-\displaystyle\frac{\alpha^1\beta^2
                p-\alpha^0\beta^3}{2\alpha^3\beta^2}\right)=0.
            \end{equation}
        This turns out to give rise to an algebraic characterisation
        of canonical discrete analogues of Godeaux-Rozet and
        Demoulin surfaces as discussed in Section \ref{secgr}. Here, we merely observe that
        if $D^1=0$ and $\bX$ is a generic point
        then \eqref{eqgr1} implies that
            \begin{equation*}
                ps_{1}=s^*,\quad s^* = \displaystyle\frac{\alpha^1\beta^2
                p-\alpha^0\beta^3}{2\alpha^3\beta^2},
            \end{equation*}
        which
                is a root of \eqref{eqsharegen} and, hence, $\bX_{1}$
                lies on the shared generator of $Q$ and $Q_{1}$. On the other hand, \eqref{eqgr1} may also be satisfied by setting $s=s^*$ and, hence, $\bX$ lies on the shared generator.
\end{remark}

    The expressions \eqref{eqs2} and \eqref{eqs22} for
    $s_{1}$ and $s_{2}$ respectively are independent of $t$. Similarly, by
    virtue of \eqref{eqt2} and \eqref{eqt22}, $t_{1}$ and $t_{2}$ are independent
    of $s$. Thus, we have come to the following conclusion.
    \begin{corollary}\label{cormapgnr}
        Let $\Sigma$ be a discrete asymptotic net. Then, the tangency condition maps generators to generators in the sense that
        generic points on a generator of $Q$ are mapped to points on a generator of
        $Q_{1}$ (of the same type).
    \end{corollary}

\subsection{Discrete projective minimal surfaces}

In analogy with the continuous theory, we are interested in
investigating the existence and properties of envelopes of a set of
lattice Lie quadrics associated with a discrete asymptotic net
$\Sigma$. Let $Q$ be a quadric associated with a quadrilateral of
$\Sigma$ and choose a generic point $\bX$ on $Q$. Then, by Theorem~\ref{thmuniqpt}, the tangency condition
uniquely determines points $\bomega_1$ and $\bomega_2$ on the neighbouring quadrics $Q_1$ and $Q_2$ respectively. 
Subsequent imposition of the tangency condition now
generates two points $\bX_{12}$ and $\bX_{21}$ on $Q_{12}$ and these are required to coincide if they are to be part of an envelope.
    \begin{figure}[h]
        \centering
        \includegraphics[scale=1]{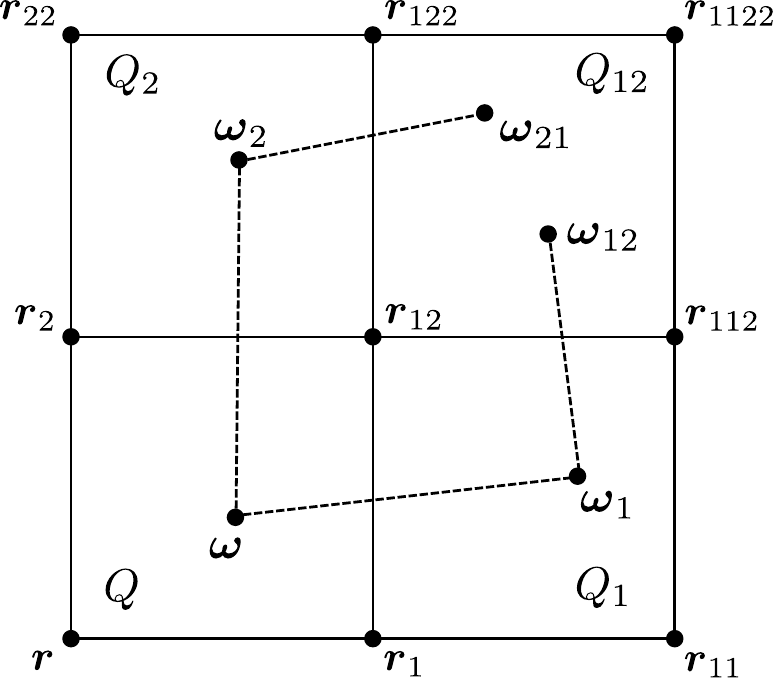}
        \caption{(Non-)closure under the tangency condition}
        \label{fignonclose}
    \end{figure}
    \begin{theorem}\label{thmpmcond}
        Let $\Sigma$ be a $3\times 3$ patch of a discrete asymptotic
        net, $Q,Q_{1},Q_{2}$ and $Q_{12}$ lattice Lie quadrics of
        $\Sigma$ as in Figure \ref{fignonclose}. Let $\bX$ be a generic point on $Q$ and $\bX_{1},\bX_{2},\bX_{12},\bX_{21}$ be points on the
        respective quadrics related by the tangency condition. Then, the commutativity
        conditions $\mathcal{S}_1^2\circ\mathcal{S}^1=\mathcal{S}_2^1\circ\mathcal{S}^2$
        and
        $\mathcal{T}_1^2\circ\mathcal{T}^1=\mathcal{T}_2^1\circ\mathcal{T}^2$
        associated with the closing condition $\bX_{21}=\bX_{12}$
        coincide and impose one scalar constraint on $\Sigma$ which
        is independent of the choice of~$\bX$.
    \end{theorem}
In order to derive a compact form of the above constraint, we require
    the following lemma.
    \begin{lemma}\label{lemgauge1}
        Let $\bV$ be the frame of a discrete asymptotic net given by
        \eqref{eqfrm1}. Then, $\bV$ admits a gauge such that
            \begin{equation*}
                (1-p^2)\displaystyle\frac{\beta^2\delta^1}{\beta^1\delta^2}=1.
            \end{equation*}
    \end{lemma}
    \begin{proof}
        Under a gauge transformation
            \begin{equation*}
                \bV\rightarrow\bV_g=\left(
                    \begin{array}{c}
                        x\mb{r}\\
                        x_{1}\mb{r}_{1}\\
                        x_{2}\mb{r}_{2}\\
                        x_{12}\mb{r}_{12}
                    \end{array}
                    \right)
            \end{equation*}
the matrices $L$ and $M$ become
            \begin{align*}
                &L_g=\left(
                    \begin{array}{cccc}
                        0&1&0&0\\
                        \displaystyle\frac{x_{11}\alpha^0}{x}&
                        \displaystyle\frac{x_{11}\alpha^1}{x_{1}}&0&
                        \displaystyle\frac{x_{11}\alpha^3}{x_{12}}\\[3mm]
                        0&0&0&1\\
                        0&\displaystyle\frac{x_{112}\beta^1}{x_{1}}&\displaystyle\frac{x_{112}\beta^2}{x_{2}}&\displaystyle\frac{x_{112}\beta^3}{x_{12}}
                    \end{array}
                    \right)\\[2mm]
                &M_g=\left(
                    \begin{array}{cccc}
                        0&0&1&0\\
                        0&0&0&1\\
                        \displaystyle\frac{x_{22}\gamma^0}{x}&0&
                        \displaystyle\frac{x_{22}\gamma^2}{x_{2}}&
                        \displaystyle\frac{x_{22}\gamma^3}{x_{12}}\\[4mm]
                        0&\displaystyle\frac{x_{122}\delta^1}{x_{1}}&\displaystyle\frac{x_{122}\delta^2}{x_{2}}&\displaystyle\frac{x_{122}\delta^3}{x_{12}}
                    \end{array}
                    \right).
                \end{align*}
        The parameters of the quadric $Q$ transform according to
        \begin{equation*}
            p_g=\left(\displaystyle\frac{x_{1}x_{2}}{xx_{12}}\right)p,\quad
            s_g=\displaystyle\frac{x_{2}}{x}s,\quad
            t_g=\displaystyle\frac{x_1}{x}t.
        \end{equation*}
    The expression $(1-p_g^2)\beta_g^2\delta_g^1/(\beta_g^1\delta_g^2)=1$
    then becomes
        \begin{equation}\label{eqgauge1}
            \left[1-\left(\displaystyle\frac{x_{1}x_{2}}{xx_{12}}\right)^2
            p^2\right]\displaystyle\frac{\beta^2\delta^1}{\beta^1\delta^2}=1
        \end{equation}
    and ``solving'' this expression for the gauge function $x$ defines the gauge.
    \end{proof}
    We now present the proof of Theorem
    \ref{thmpmcond}.
    \begin{proof}
    Set $\bX=Q(s,t)$, $\bX_1=Q_1(s_1,t_1)$, $\bX_2=Q_2(s_2,t_2)$, where
    the expressions for $s_1,t_1,s_2$ and $t_2$ are given by
    \eqref{eqs2}, \eqref{eqt2}, \eqref{eqs22} and \eqref{eqt22}
    respectively. The two points
    $\bX_{12}=Q_{12}(s_{12},t_{12})$ and $\bX_{21}=Q_{12}(s_{21},t_{21})$ arise from the tangency condition associated with $Q_{12}$ and $Q_1$ and $Q_{12}$ and
    $Q_2$ respectively. The closing condition is encapsulated in  the two constraints
    $s_{12}=s_{21}$ and $t_{12}=t_{21}$ or, equivalently, $\mathcal{S}_1^2\circ\mathcal{S}^1(s)=\mathcal{S}_2^1\circ\mathcal{S}^2(s)$
    and $\mathcal{T}_1^2\circ\mathcal{T}^1(t)=\mathcal{T}_2^1\circ\mathcal{T}^2(t)$.
    A somewhat tedious calculation reveals that these conditions coincide and, utilising the gauge of Lemma \ref{lemgauge1},
    the closing conditions may then be formulated as the condition
        \begin{equation}\label{eqpmcond}
            \Delta_2 T^1:=T_2^1 - T^1 = 0
        \end{equation}
    or, equivalently,
        \begin{equation}\label{eqpmcond2}
           \Delta_1 T^2:= T_1^2 - T^2 = 0,
        \end{equation}
    where
        \begin{equation*}
            T^1=\displaystyle\frac{D^1}{(\alpha^0\beta^2p)^2},\quad
            T^2=\displaystyle\frac{D^2}{(\gamma^0\delta^1 p)^2},
        \end{equation*}
    and $D^1$ and $D^2$ are the discriminants \eqref{eqdisc} and \eqref{eqd2}.
    The equivalence of these algebraic conditions may be seen by
    extracting the relation
        \begin{equation}\label{keydisc}
            \frac{\beta^1}{\beta^2}\Delta_1T^2=\frac{\delta^2}{\delta^1}\Delta_2T^1
        \end{equation}
    from the discrete Gauss-Mainardi-Codazzi equations \eqref{eqcomp1}. Finally, we note that the algebraic constraint \eqref{eqpmcond} (or \eqref{eqpmcond2}) is independent of $s$ and $t$,
    i.e., independent of the choice of $\bX$.
    \end{proof}

\begin{remark}
   Relation \eqref{keydisc} is the discrete analogue of the Gauss-Mainardi-Codazzi equation \eqref{gmc}
and the equivalent constraints \eqref{eqpmcond}, \eqref{eqpmcond2} may be regarded as the discrete version of the Euler-Lagrange condition \eqref{eulerlagrange} in the classical Theorem \ref{elcond}.
\end{remark}

The above theorem gives rise to the following natural algebraic definition of discrete projective minimal surfaces.

    \begin{definition}
        A discrete asymptotic net is a discrete
        projective minimal surface if $\Delta_1 T^2=0$
        or, equivalently, $\Delta_2 T^1=0$ in the gauge of Lemma
        \ref{lemgauge1}.
    \end{definition}

A detailed discussion of the discrete system \eqref{eqcomp1}, \eqref{eqpmcond} (or \eqref{eqpmcond2}) underlying discrete projective minimal surfaces, including the determination of its integrability and an analogue of the algebraic classification scheme presented in Section 2, is the subject of a separate publication.

\section{A Cauchy problem}
The algebraic discrete projective minimality condition may be exploited to
state the following well-posed geometric Cauchy problem.
    \begin{theorem}\label{thmcauchy}
        A discrete projective minimal surface $\Sigma$ represented by\linebreak \mbox{$\mb{r}:\mathbb{Z}^2\rightarrow\mathbb{R}^4$} is uniquely determined
        by the Cauchy data $\{\mathscr{C},\,Q_0\}$, where
            \begin{equation*}
                \mathscr{C}=\{\mb{r}(\mb{n}):\,\mb{n}=(0,*),\,(1,*),\,(*,0),\,(*,1)\}
            \end{equation*}
        is constrained by the planar star property and $Q_0$ is a fixed quadric associated with one of the quadrilaterals of $\mathscr{C}$ (Figure \ref{figcauchy}).
    \end{theorem}
    \begin{proof}
        We note that the Cauchy data given by $\mathscr{C}$ define part of a discrete asymptotic net consisting of
        two strips as in Figure \ref{figcauchy}. Without loss of
        generality, let $Q_0$ be the quadric associated with the
        quadrilateral
        $[\mb{r},\mb{r}_{1},\mb{r}_2,\mb{r}_{12}]$.
        As a result of the $\mathcal{C}^1$ condition, $Q_0$ uniquely
        determines the quadric associated with each quadrilateral
        along the strips. We now show that the vertex $\mb{r}_{1122}$ is uniquely determined by the projective minimality condition.
        The planar star condition implies that the point $\mb{r}_{1122}$ lies
        on the line of intersection of the planes spanned by
        $\{\mb{r}_{12},\mb{r}_{11},\mb{r}_{112}\}$ and
        $\{\mb{r}_{12},\mb{r}_{22},\mb{r}_{122}\}$ as indicated in Figure \ref{figcauchy}. For a fixed
        $\mb{r}_{1122}$ on this line, the quadric
        $Q_{12}$ is uniquely determined by the $\mathcal{C}^1$
        condition with respect to $Q_1$ or $Q_2$. Let $\bX\in Q_0$ be a
        generic point. Then, the tangency condition uniquely
        determines points $\bX_1\in Q_1$, $\bX_2\in Q_2$ and
        $\bX_{12},\,\bX_{21}\in Q_{12}$. As shown in Theorem \ref{thmpmcond}, the closing condition
        $\bX_{12}=\bX_{21}$ reduces to a single scalar condition and this
        determines the position of $\mb{r}_{1122}$ on the line of
        intersection.  We can then iterate this procedure to
        construct simultaneously all vertices of the discrete asymptotic net, the
        associated family of lattice Lie quadrics and the associated
        envelope.
    \end{proof}
    \begin{figure}[h]
        \centering
        \includegraphics[scale=1]{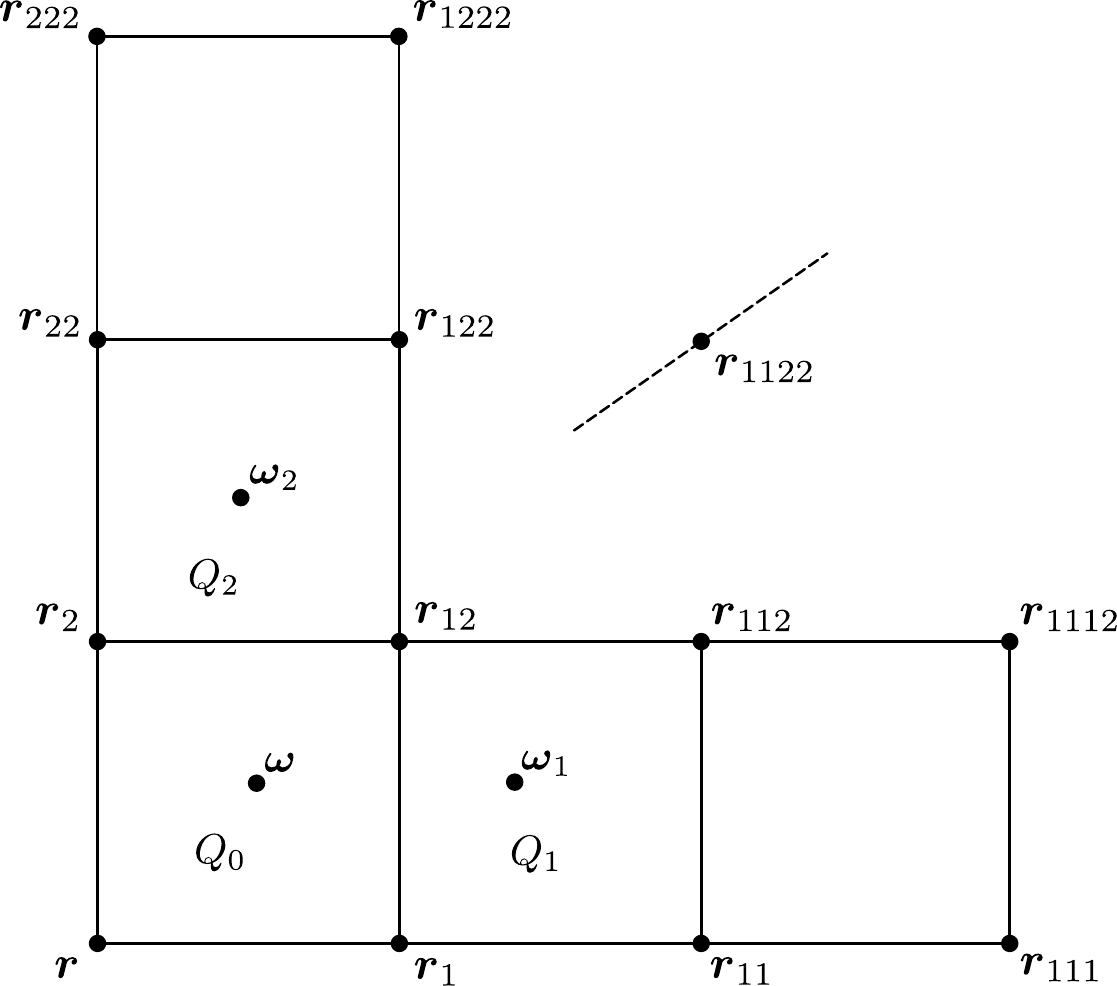}
        \caption{The evolution of Cauchy data for discrete projective minimal surfaces}
        \label{figcauchy}
    \end{figure}
    Since the Cauchy data include a fixed quadric which in turn determines a set of lattice Lie quadrics, the natural question arises as to
    whether a given discrete projective minimal surface admits more
    than one set of lattice Lie quadrics.
    \begin{theorem}\label{thm3}
        A discrete projective minimal surface admits
        only one set of lattice Lie quadrics.
    \end{theorem}
    \begin{proof}
        Consider the $4\times 3$ patch of a discrete projective minimal surface in Figure
        \ref{fig4x3patch}. Let $Q$ be a fixed member, labelled by $p$, of the
        one-parameter family of quadrics passing through the
        quadrilateral
        $[\mb{r},\mb{r}_{1},\mb{r}_{2},\mb{r}_{12}]$. We note
        that, as a result of the $\mathcal{C}^1$ condition, all other
        quadrics in the set of lattice Lie quadrics which contains
        $Q$ are determined. Let $\bX$ be a generic point on $Q$.
        The tangency condition then delivers points $\bX_1\in Q_{1},\,\bX_2\in Q_2$ 
        and $\bX_{12},\bX_{21} \in Q_{12}$ and the closing condition $\bX_{12}=\bX_{21}$
        determines $\mb{r}_{1122}$ uniquely. It turns out that
        $\mb{r}_{1122}$ is a quadratic function of $p$. If we now assume that
        $\mb{r}_{1122}(p)\sim\mb{r}_{1122}(\hat{p})$ for another quadric $\hat{Q}$ 
        labelled by $\hat{p}$ then we obtain a
        quadratic equation in $\hat{p}$ which we denote by $q(\hat{p})=0$. Similarly, the
        vertex $\mb{r}_{\bar{1}22}$ generates a quadratic
        $\bar{q}(\hat{p})=0$. It is then easy to check that the only
        common root of $q$ and $\bar{q}$ is $p$ and, hence, $\hat{Q}=Q$.
    \end{proof}
    \begin{figure}[h]
        \centering
        \includegraphics[scale=1]{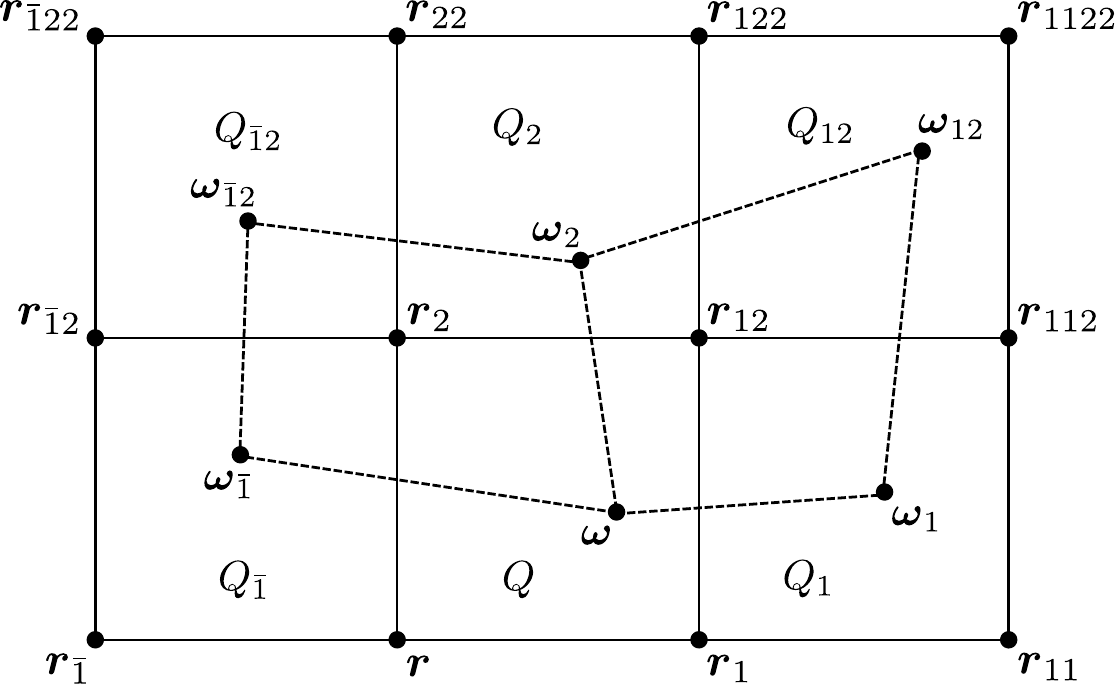}
        \caption{A $4\times3$ patch of a discrete projective minimal surface}\label{fig4x3patch}
    \end{figure}

\section{A local classification of the envelopes of discrete PMQ surfaces}\label{seclclass}
\label{local}

In order to produce a geometric characterisation of discrete
projective minimal surfaces, it is necessary to classify all surfaces
which admit envelopes, i.e., by definition, it is necessary to classify discrete PMQ
surfaces. Let $\Sigma$ be a discrete PMQ surface and denote by
$\Omega$ an envelope of the set of associated lattice Lie quadrics
of $\Sigma$. Bearing in mind the continuum limit, from now on we
exclude ``hybrids" of different types of surfaces by imposing a
homogeneity condition on the envelope in the sense that if a
property related to the envelope holds for pairs of neighbouring
quadrics then it holds for all neighbouring quadrics of the same
type. In the following, whenever this principle is applied, the relevant property is identified.
    \begin{figure}[h]
        \centering
          \includegraphics[scale=1]{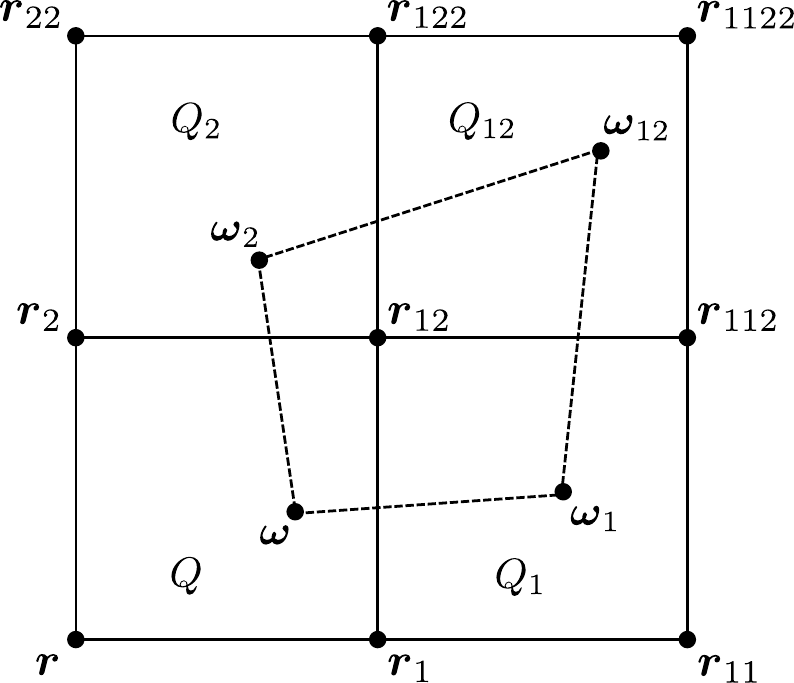}
        \caption{A $3\times3$ patch of a discrete asymptotic net and an associated quadrilateral of an enevelope}
        \label{fig3x3patch}
    \end{figure}

In the examination of different types of envelopes, the
maps $\mathcal{S}^i$ and $\mathcal{T}^i$ play an essential role.
Consider a $3\times 3$ patch of $\Sigma$ as displayed in Figure~\ref{fig3x3patch}
and the associated quadrilateral of $\Omega$. The vertices of $\Omega$ are labelled by
\mbox{$\bX=Q(s,t)$}, $\bX_1=Q_1(s_1,t_1)$, $\bX_2=Q_2(s_2,t_2)$ and
$\bX_{12}=Q_{12}(s_{12},t_{12})$. Depending on whether or not the maps $\mathcal{S}^i$ and $\mathcal{T}^i$ are defined, whereby the homogeneity condition is taken into account, the possible cases are:
    \begin{enumerate}[label=\arabic*.]
        \item
        The generic case where all maps $\mathcal{S}^i$ and $\mathcal{T}^i$ are defined and, hence, the closing conditions are satisfied due to the existence of the discrete envelope $\Omega$. $\Sigma$ is then discrete projective
        minimal by Theorem \ref{thmpmcond}. This is the case where
        none of the edges of the quadrilateral of $\Omega$ are shared
        generators of the associated pairs of quadrics.
        \item
        The map $\mathcal{S}^2$ is defined. There are then two subcases:
            \begin{enumerate}[label=(\alph*), ref=\theenumi{}(\alph*)]
                \item
                $\mathcal{S}^1$ is defined. Then, since $\Omega$ is a discrete envelope, the closing condition
                $\mathcal{S}_1^2\circ\mathcal{S}^1=\mathcal{S}_2^1\circ\mathcal{S}^2$ is satisfied and, hence, $\Sigma$ is discrete projective
                minimal.\label{case2a}
                \item
                $\mathcal{S}^1$ is not defined. By virtue of \eqref{eqtang4}, \eqref{eqs2}, this is the case when  \mbox{$2\alpha^3\beta^2 s  - \alpha^1\beta^2 p+\alpha^0\beta^3=0$} 
                and, hence,
                    \begin{equation}\label{eqsshare1}
                        s=\displaystyle\frac{\alpha^1\beta^2
                        p-\alpha^0\beta^3}{2\alpha^3\beta^2}
                    \end{equation}
                and $D^1=0$. Accordingly, $\Delta_2 T^1=0$ and, thus, $\Sigma$ is discrete projective minimal.
                Moreover, since $D^1=0$, by Corollary \ref{corsharegen1}, neighbouring quadrics in the $n_1$ direction have a coinciding shared generator which is labelled
                by
                $ps_1=s$ and \eqref{eqsshare1}. Thus, this case corresponds to the edges of $\Omega$ in the $n_1$ direction consisting of (coinciding)
                shared generators
                of neighbouring pairs of quadrics.\label{wsim}
            \end{enumerate}
        \item
        $\mathcal{S}^2$ is not defined. By
                \eqref{eqttanga}, this is the case when $pt_2=t$ and, hence, by Remark \ref{remarsgen}, the edges of $\Omega$ in the $n_2$
                direction are shared generators. The remaining cases
        are then:
            \begin{enumerate}[label=(\alph*)]
                \item
                $\mathcal{T}^2$ is defined.  There are then two additional sub-cases:
                    \begin{enumerate}[label=(\roman*)]
                        \item
                        $\mathcal{T}^1$ is defined. Then similarly to \ref{case2a}, $\Sigma$ is discrete
                        projective minimal as the projective minimality
                        condition $\mathcal{T}_1^2\circ \mathcal{T}^1=\mathcal{T}_2^1\circ
                        \mathcal{T}^2$ is satisfied.
                        \item
                        $\mathcal{T}^1$ is not defined. By \eqref{eqtang3} and \eqref{eqt2}, $\mathcal{T}^1$ not being well defined corresponds to the case
                        $ps_1=s$
                        which, by the proof of Theorem \ref{thmsharegen1}, implies that the edges of $\Omega$ in the $n_1$ direction are shared generators of neighbouring quadrics.
                        Thus, this case corresponds
                        to all of the edges of the envelope $\Omega$ consisting
                         of shared generators. This envelope always
                        exists locally, that is, for a $3\times 3$ patch, and is no restriction on $\Sigma$
                        which may or may not be discrete projective minimal.
                    \end{enumerate}
                \item
                $\mathcal{T}^2$ is not defined. This case is similar to
                \ref{wsim} and leads to $D^2=0$. $\Sigma$ is then discrete projective minimal since
                $\Delta_1 T^2=0$.
            \end{enumerate}
    \end{enumerate}

\section{A classification of discrete PMQ surfaces}
As a result of the discussion of Section \ref{local}, we have the following geometric characterisation of discrete projective minimal surfaces.
    \begin{theorem}\label{thmgeopm}
        Let $\Sigma$ be a discrete asymptotic net. If there exists a set of lattice Lie quadrics which admits an
        envelope $\Omega$ whose edges are not all shared generators of the
        lattice Lie quadrics then $\Sigma$ is a discrete projective minimal
        surface. Conversely, if $\Sigma$ is a discrete projective minimal surface and $D^1\neq0$ or $D^2\neq0$ then there exists an
        envelope $\Omega$ whose edges are not all shared generators of the
        lattice Lie quadrics.
    \end{theorem}
    \begin{proof}
        By the discussion of Section \ref{local}, if  $\Omega$ does not entirely consist of shared generators then $\Sigma$ is discrete projective minimal.
        Conversely, suppose that $\Sigma$ is discrete projective minimal. In the general case $D^1\neq0$ and $D^2\neq0$, any choice of a generic point $\bX\in Q$ gives rise to an envelope $\Omega$ since the maps $\mathcal{S}^i$ and $\mathcal{T}^i$
        in \eqref{eqs2}, \eqref{eqt2}, \eqref{eqs22} and \eqref{eqt22} are defined and compatible by virtue of the algebraic minimality condition. Due to the genericity of $\bX$, the edges of $\Omega$ are not shared generators. If, for instance, $D^1=0$ but $D^2\neq0$ then the map $\mathcal{S}^1$ simplifies to (cf.\ \eqref{eqs2})
\begin{equation}\label{maps1}
  \mathcal{S}^1(s) = s_1 = \frac{\alpha^1\beta^2p-\alpha^0\beta^3}{2\alpha^3\beta^2 p}
\end{equation}
(which is, in fact, independent of $s$), provided that
\begin{equation}\label{maps}
  s\neq \frac{\alpha^1\beta^2p-\alpha^0\beta^3}{2\alpha^3\beta^2}
\end{equation}
and the maps $\mathcal{S}^i$ and $\mathcal{T}^i$ are still compatible. Even if the condition \eqref{maps} is violated then we may regard \eqref{maps1} as a definition of the map $\mathcal{S}^1$ (with the associated tangency condition \eqref{eqtang4} being satisfied) and the maps $\mathcal{S}^i$ and $\mathcal{T}^i$ remain compatible since the constraint \eqref{maps} does not enter the compatibility condition. In this case, both $\bX$ and $\bX_1$ lie on the generator shared by $Q$ and $Q_1$ but regardless of whether \eqref{maps} holds or not, since the (initial) parameter $t$ is arbitrary, we may choose it in such a manner that the edges of the corresponding envelope $\Omega$ in $n_2$ direction are not shared generators.
\end{proof}

    \begin{remark}\label{remgenenv}
        If we refer to envelopes which do not consist entirely of
        shared generators as being generic, then the above theorem
        implies that, in general, discrete projective minimal
        surfaces admit a two-parameter family of generic envelopes.
    \end{remark}
While Theorem \ref{thmgeopm} characterises all discrete projective minimal surfaces, there also exist surfaces with an associated set of lattice Lie quadrics which
admit envelopes but may not be discrete projective minimal.
    \begin{definition}\label{defqs}
        A discrete Q surface is a discrete asymptotic net which admits an envelope $\Omega$ whose (extended) edges are shared generators of the associated lattice Lie quadrics so that the coordinate polygons of $\Omega$ are straight lines.
    \end{definition}
    \begin{figure}[h]
        \centering
        \includegraphics[scale=1]{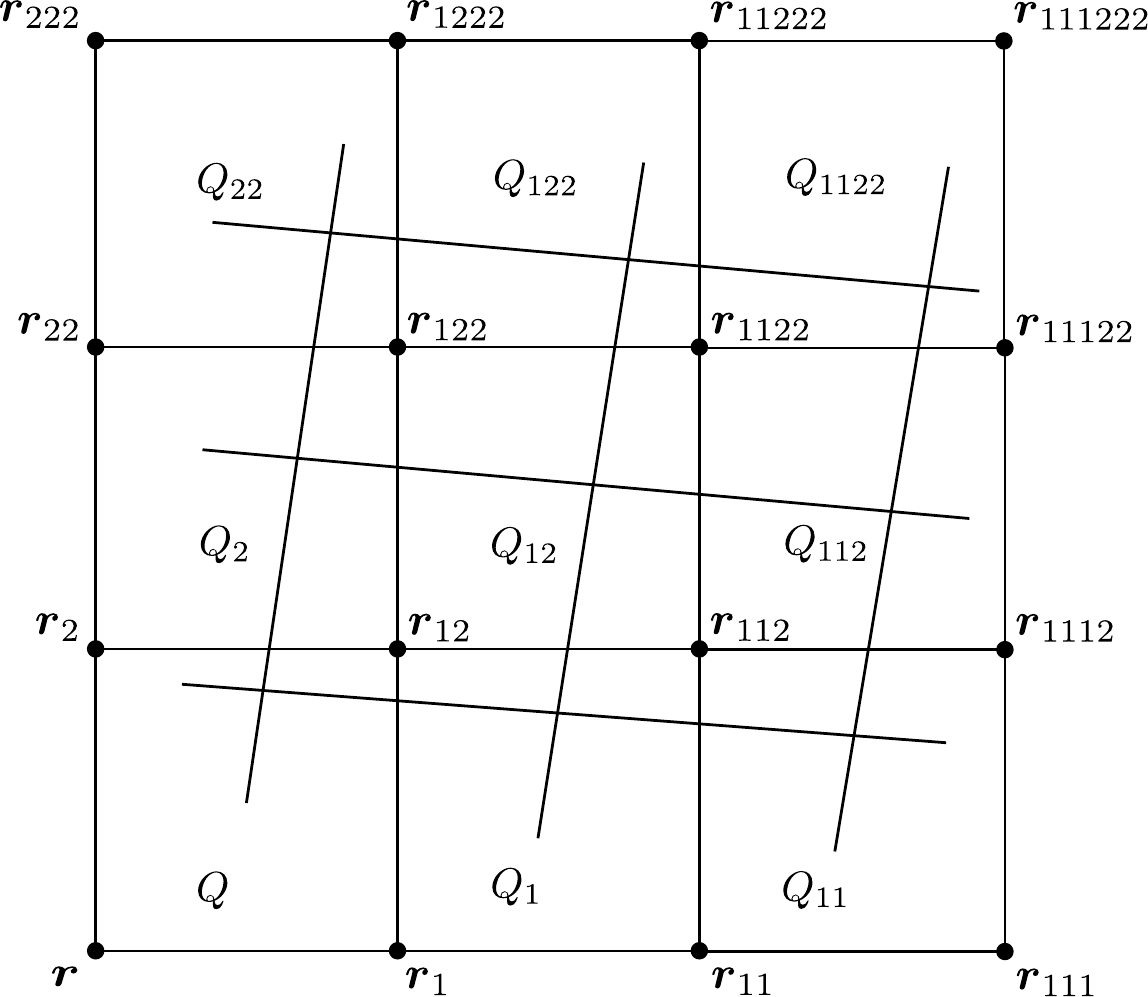}
        \caption{A discrete Q surface and the associated envelope}
        \label{figqsurface}
    \end{figure}
    \begin{remark}
        It follows from the definition of a Q surface that each straight coordinate polygon of the given envelope is a generator common to all quadrics along a strip of the
        Q surface.
        In analogy with the continuous theory, this envelope is a discrete quadric as it consists of two $1$-parameter (discrete) families of lines. In fact, it is evident that their exists a unique
(continuous) quadric which passes through those lines.
        If a discrete Q surface admits an envelope which is not entirely made up of shared generators then, by Theorem \ref{thmgeopm}, it is
        discrete projective minimal. Thus, the classes of discrete projective minimal and discrete Q surfaces are, {\it a priori}, not disjoint.
    \end{remark}
In order to establish in what sense the class of discrete PMQ
surfaces consist only of discrete projective minimal and discrete
Q surfaces, we make use of the fact that if there exists an envelope
whose edges along a strip are shared generators then these shared
generators are identical. Hence, by the homogeneity condition, if
two quadrics $Q$ and $Q_1$ have a (coinciding) shared generator which
forms part of an envelope then all quadrics along that strip are
assumed to have the same (coinciding) generator. Then, as a result
of Theorem \ref{thmgeopm}, Definition \ref{defqs} and the
homogeneity condition, we can draw the following conclusion.
    \begin{corollary}
        The class of discrete PMQ surfaces consists of discrete projective minimal and discrete Q surfaces.
    \end{corollary}

\section{Discrete semi-Q, complex, doubly Q and doubly complex surfaces}
We now investigate in more detail the geometric nature of Q surfaces and classes intimately related to Q surfaces. The elementary building block of the notion
of a Q surface is the notion of a semi-Q surface.                                                                                                                                                                                                                                                                                                                                                                                                                                                                                                                                                                                                                                                                                                                                                                                                                                                                                                                                                                                                                                                                                                                                                                                                                                                                                                                                                                                                                                                                                                                                                                                                                                                                                                                                                                                                                                                                                                                                                                                                                                                                                %
    \begin{definition}
        A discrete asymptotic net is said to be a discrete semi-Q 
        surface if it admits a set of lattice Lie quadrics such
        that quadrics associated with each strip of a given type share a generator, as in the top left of Figure~\ref{figqsurfaces}.
        We call the direction ($n_1$ or $n_2$) along which quadrics share a
        generator the semi-Q direction.
    \end{definition}

We now examine surfaces which are semi-Q in more
than one way.
\subsection{Discrete complex, doubly Q and doubly complex surfaces}
    \begin{definition}
        As illustrated in Figure \ref{figqsurfaces}, a discrete asymptotic net with an associated set of lattice Lie quadrics
        is said to be a discrete
            \begin{enumerate}[label=(\roman*)]
                \item
                complex surface if all quadrics on each strip of a given type share two (possibly coinciding) generators,
                that is, if it is doubly semi-Q in one direction.
                \item
                doubly Q surface if it is a discrete complex surface
                with respect to one direction and a discrete semi-Q
                surface with respect to the other direction, that is,
                if it is doubly semi-Q in one direction and semi-Q
                in the other.
                \item
                doubly complex surface if it is a discrete complex surface
                in both directions, that is, if it is doubly semi-Q
                in both directions.
            \end{enumerate}
    \end{definition}

\begin{remark}
 By definition, discrete Q surfaces are discrete asymptotic nets which are semi-Q in both directions.
\end{remark}
    \begin{figure}[h]
        \centering
        \includegraphics[scale=0.75]{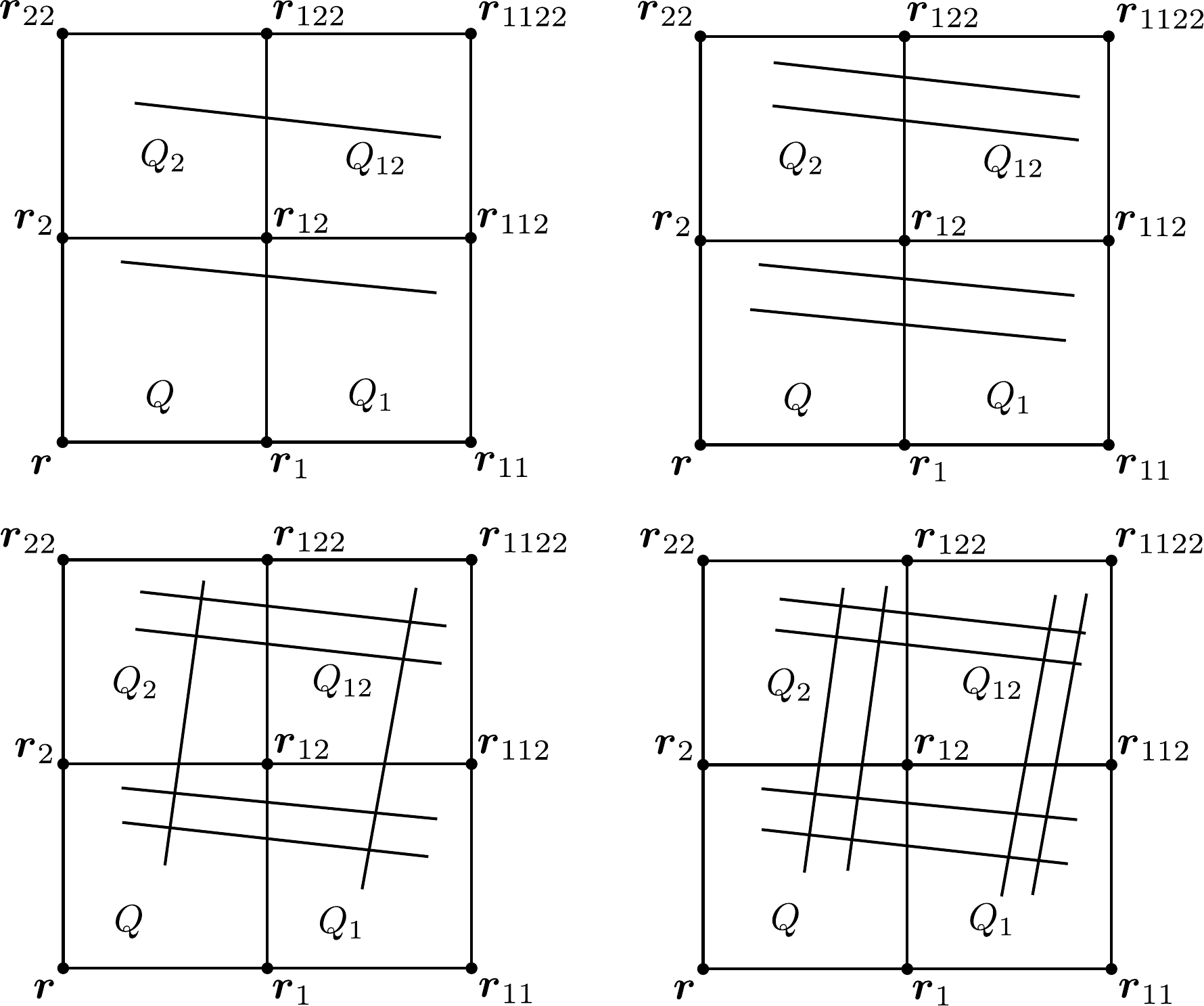}
        \caption{Discrete semi-Q (t-l), complex (t-r), doubly Q (b-l) and doubly complex (b-r) surfaces}
        \label{figqsurfaces}
    \end{figure}
    As in the continuous theory, it turns out that not all of these classes of surfaces are
    distinct. Firstly, we need the following well-known fact.
    \begin{lemma}\label{lem3pts}
        Let $\hat{Q}$ be a quadric, $L$ a line. Then,
            \begin{enumerate}[label=(\roman*)]
                \item
                if  $L$ intersects $\hat{Q}$ at three points then it is a generator of $\hat{Q}$,
                \item
                if $L$ intersects $\hat{Q}$ at a point and touches $\hat{Q}$ at a different point then it is a generator of $\hat{Q}$.
            \end{enumerate}
    \end{lemma}

    \begin{theorem}\label{thm2is4}
        A discrete doubly Q surface is doubly complex.
    \end{theorem}
    \begin{proof}
        Without loss of generality, we assume that the surface is
        discrete semi-Q in the $n_2$ direction and discrete complex
        in the $n_1$ direction. Let $\hat{Q}$ be the unique quadric
        passing through three neighbouring shared generators $W$, $W_1$ and $W_{11}$ as in Figure \ref{fig3strip}. Denote by $U$ and $V$ the generators common
        to the quadrics of the $n_1$ strip containing the quadric $Q$ as in Figure \ref{fig3strip}.
        Since $U$ and $V$ and
        their shifts in the $n_2$ direction intersect $W$, $W_1$ and $W_{11}$, by Lemma~\ref{lem3pts}, $U$, $V$ and their $n_2$
        shifts are all generators of $\hat{Q}$.
        Now, consider the three quadrics $Q,\,Q_2,$ and $Q_{\bar{2}}$ in an $n_2$ strip as in
        Figure \ref{fig3strip}. We treat the cases of coinciding and non-coinciding generators in the $n_1$ direction separately.
            \begin{enumerate}[label=(\roman*)]
                \item
                    Suppose firstly that $U\neq V$. We know by Theorem
                    \ref{thmsharegen1} that $Q$ and $Q_{2}$ have a second shared
                    generator $G$ (which may or may not coincide with $W$). $G$ intersects $U_2$, $V_2$, $U$ (and $V$) and, hence, by Lemma \ref{lem3pts}, $G$ is a
                    generator of $\hat{Q}$.
                    Thus, by extension, $G$ also intersects $U_{\bar{2}}$ and $V_{\bar{2}}$. Moreover, since the
                    edge $[\mb{r},\mb{r}_{1}]$ is a generator of $Q_{\bar{2}}$
                    and since $G$ intersects this edge as well and, thus, intersects three generators of $Q_{\bar{2}}$, by Lemma
                    \ref{lem3pts}, $G$ is also a generator of $Q_{\bar{2}}$.
                \item
                    Suppose now that $U=V$. Denote by $\bH$ the point of intersection of $U$ and $W$, by $\bH_1$ the point
                    of intersection of $U$ and $W_1$ etc. Then, since the tangent
                    plane of $Q$ at $\bH$ is spanned by $U$ and $W$, $Q$ and $\hat{Q}$ touch at $\bH$. Similarly, $\hat{Q}$ and $Q_1$ touch at $\bH_1$ and $\hat{Q}$ and
                    $Q_{11}$ touch at
                    $H_{11}$. Moreover, by Corollary \ref{cortouch}, the tangent planes of $Q$ at $\bH_1$ and $\bH_{11}$ coincide with those of $Q_1$ and $Q_{11}$
                    respectively.
                    Consequently, the tangent planes of $\hat{Q}$ and $Q$ coincide at the three points $\bH$, $\bH_1$ and $\bH_{11}$ and, therefore, $\hat{Q}$
                    touches $Q$ along $U$.
                    Then, since $G$ is a generator of $Q$ and $Q_2$, it intersects $U$ and $U_2$ at points $\bK$ and $\bK_2$, say, and, hence, touches $\hat{Q}$ at $\bK$
                    and $\bK_2$. Thus, by Lemma \ref{lem3pts}, $G$ is a generator of $\hat{Q}$. By extension, $G$
                    intersects $Q_{\bar{2}}$
                    at a point on the edge $[\mb{r},\mb{r}_{1}]$ and touches $Q_{\bar{2}}$ at $\bK_{\bar{2}}$ (being the intersection of $G$ and $U_{\bar{2}}$) since $\hat{Q}$ touches $Q_{\bar{2}}$ along $U_{\bar{2}}$. Thus,
                    $G$ is a generator of $Q_{\bar{2}}$.
            \end{enumerate}\vspace{-8mm}

    \end{proof}
    \begin{figure}[h]
        \centering
        \includegraphics[scale=1]{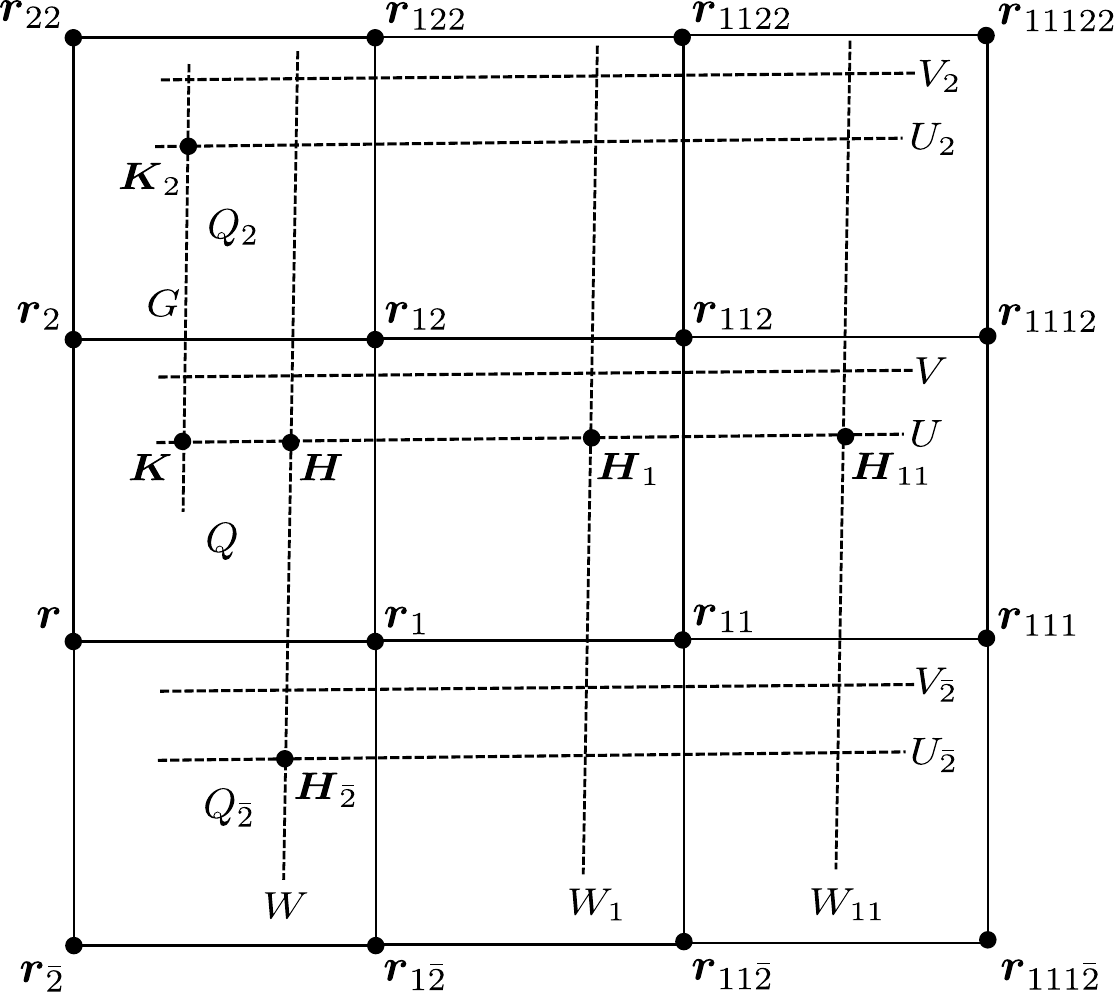}
        \caption{A $4\times 4$ patch of a discrete doubly Q surface}
        \label{fig3strip}
    \end{figure}

\section{Discrete Godeaux-Rozet, Demoulin and\\ Tzitz\'eica surfaces}\label{secgr}
We now turn to an investigation of special types of discrete projective minimal surfaces and define discrete analogues of Godeaux-Rozet, Demoulin
and Tzitz\'eica surfaces. 

\subsection{Discrete Godeaux-Rozet and Demoulin surfaces}

In view of the classification of Section \ref{seclclass}, we propose the following algebraic definition.
    \begin{definition}\label{defngr}
        A discrete asymptotic net $\Sigma$ is a
            \begin{enumerate}[label=(\roman*)]
                \item
                discrete Godeaux-Rozet surface if $D^1=0$ or $D^2=0$.
                \item
                discrete Demoulin surface if
                $D^1=D^2=0$.
            \end{enumerate}
    \end{definition}
    Let $\Sigma$ be a discrete PMQ surface and $\Omega$ an envelope
    corresponding to a set of lattice Lie quadrics of $\Sigma$.
    We may associate with each vertex $\bX$ of $\Omega$ four lines passing through $\bX$ and the four vertices of the corresponding quadrilateral of
    $\Sigma$.
    Taking one of these lines and its shifts along the lattice generates a discrete line congruence. Hence, there exist four such discrete line congruences
    $\mathscr{L}$, $\mathscr{L}^1$, $\mathscr{L}^2$ and $\mathscr{L}^{12}$ as displayed in Figure \ref{figlineconarrows}.
    A priori, generically, there exists a two-parameter family of each line
    congruence generated by moving the point $\bX$ around on a fixed lattice Lie quadric.
    \begin{figure}[h]
        \centering
        \includegraphics[scale=1]{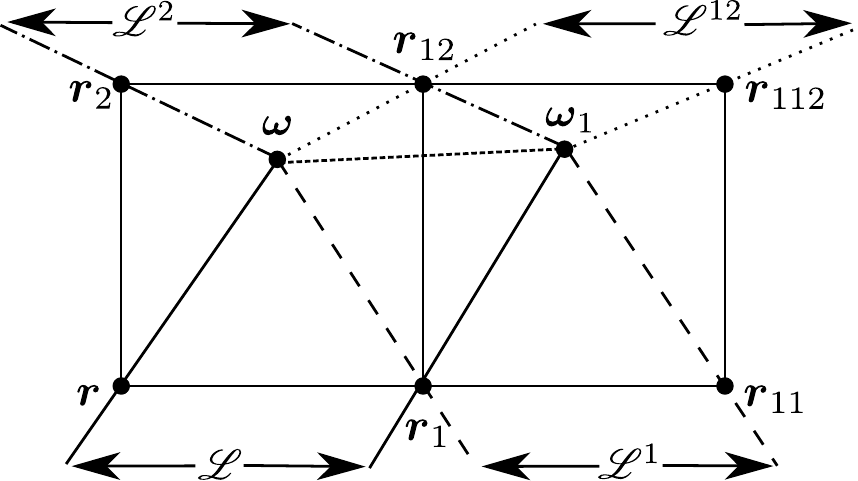}
        \caption{Four types of line congruences}
        \label{figlineconarrows}
    \end{figure}

In order to produce a geometric characterisation of discrete Godeaux-Rozet and Demoulin surfaces in terms of these line congruences, it is convenient to introduce the following definition.

\begin{definition}
A discrete line congruence $\mathscr{L}^*=\{l(n_1,n_2)\}$ is said to
have the intersection property in the direction $n_1$ (or $n_2$) if
neighbouring lines in the $n_1$ (or $n_2$) direction intersect, that
is, $l\cap l_1\neq\emptyset$ (or $l\cap l_2\neq\emptyset)$.
\end{definition}

\begin{theorem}\label{thmgeogr1}
        Let $\Sigma$ be a discrete projective minimal surface with an associated discrete envelope $\Omega$ such that the edges of $\Omega$ in the $n_1$
        direction are not shared generators of the lattice Lie quadrics of $\Sigma$.
        Then, $D^1=0$ if and only if there exists a line congruence which possesses the intersection property in the $n_1$ direction. Moreover, if $\mathscr{L}$ (or $\mathscr{L}^1$)
        admits such an intersection property then so does $\mathscr{L}^2$ (or $\mathscr{L}^{12}$) and vice versa. These statements apply, {\it mutatis mutandis}, if one considers the direction $n_2$.
    \end{theorem}
    \begin{proof}
        Consider the patch of $\Sigma$ in Figure \ref{figgeogr}. Since the edges of $\Omega$ in the $n_1$ direction are not shared generators, given two vertices of
        $\Omega$, $\bX\in Q$ and $\bX_1\in Q_1$, it follows that either $\bX$ or $
\bX_1$ does not lie on a generator common to $Q$ and $Q_1$.
        Without loss of generality, we assume that $\bX$ does not lie on a shared generator. Then, $\bX_{1}=Q_{1}(s_{1},t_{1})$, where $s_{1}$ and $t_{1}$
        are given by \eqref{eqs2} and \eqref{eqt2}. We now consider the line congruence $\mathscr{L}$ (defined by the line through $\bX$ and \mb{r} and its shifts) and
        assume that the lines $l = [\bX,\mb{r}]$ and $l_1 = [\bX_1,\mb{r}_1]$ intersect. Their point of intersection,
        $\bI^1$, is then determined by $f\mb{r}+g\mb{r}_{1}+h
        \bX+\bX_{1}=0$, that is,
            \begin{equation}\label{eqgr2}
                \begin{split}
                    \mb{r}(f+hst+\alpha^0
                    s_{1})&+\mb{r}_{1}(g+hs+\beta^1 p_{1}+\alpha^1
                    s_{1}+s_{1}t_{1})\\
                    &+\mb{r}_{2}(h+\beta^2
                    p_{1})+\mb{r}_{12}(hp+\beta^3 p_{1}+\alpha^3
                    s_{1}+t_{1})=0.
                \end{split}
            \end{equation}
        Equating each component equal to zero then yields four conditions which may be formulated as:
            \begin{align*}
                \mb{r}:\qquad &
                f=\alpha^0\left(\frac{s}{p}-s_{1}\right)\\[2mm]
                \mb{r}_1:\qquad &
                g=\frac{\alpha^0\beta^1}{\beta^2 p}-\frac{\alpha^0 s}{p t}+\alpha^1 s_1+s_1 t_1\\[2mm]
                \mb{r}_2:\qquad &
                h=-\frac{\alpha^0}{p t}\\[2mm]
                \mb{r}_{12}:\qquad &
                (\alpha^0\beta^3-\alpha^1\beta^2 p)^2+4\alpha^0\alpha^3\beta^1\beta^2
                p=0,\quad \mbox{i.e.,}\quad D^1=0.
            \end{align*}
    Conversely, if $D^1=0$ then choosing $f$, $g$ and $h$ as above implies that \eqref{eqgr2} holds so that $\bI^1$ exists. By symmetry, the above arguments also hold for
    the line congruence $\mathscr{L}^2$.
    \end{proof}
    \begin{figure}[h]
        \centering
        \includegraphics[scale=1]{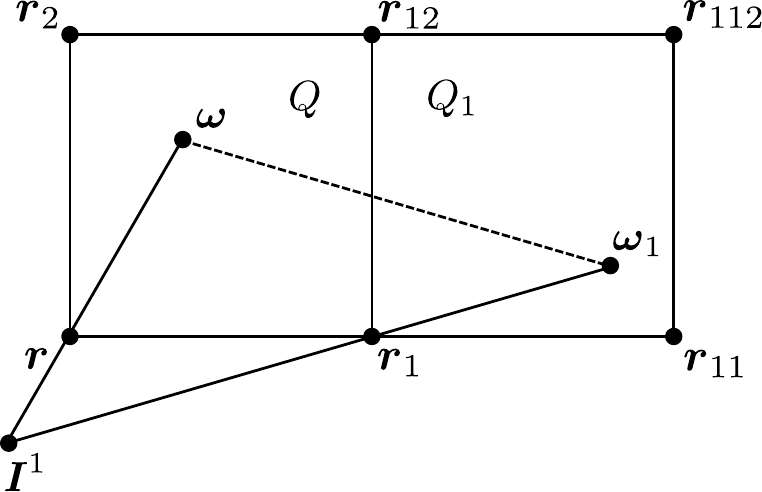}
        \caption{Intersection of lines belonging to the line congruence $\mathscr{L}$}
        \label{figgeogr}
    \end{figure}
As a result of Theorem \ref{thmgeogr1}, it follows that discrete
Godeaux-Rozet surfaces are essentially characterised by the
intersection property of their line congruences.
    \begin{corollary}
        Under the
        assumption of Theorem \ref{thmgeogr1}, a discrete projective minimal
        surface is a discrete Godeaux-Rozet surface if and only if there
        exists a line congruence associated with any envelope which possesses the intersection property in the $n_1$ or $n_2$ direction.
    \end{corollary}
    \begin{remark}\label{remgr}
        In the proof of Theorem \ref{thmgeogr1}, since $\bX$ does not lie on a shared generator, $D^1=0$ implies by virtue of \eqref{eqtang4} that
            \begin{equation*}
                s_{1}=\displaystyle\frac{\alpha^1\beta^2
                p-\alpha^0\beta^3}{2\alpha^3\beta^2 p}
            \end{equation*}
        and, hence, by Remark \ref{remarsgen}, $\bX_{1}$ is on the shared generator of $Q$ and $Q_1$.
        $\bI^1$ is then given by
            \begin{equation}\label{I1}
                \bI^1=\bX+\frac{f}{h}\mb{r}=p\mb{r}_{12}+s\mb{r}_{1}+t\mb{r}_{2}+ps_{1}t\mb{r}.
            \end{equation}
        On the other hand, if $\bX_1$ does not lie on the generator common
        to $Q$ and $Q_1$ then $D^1=0$ implies that $\bX$ lies on the shared generator. In this case, the line congruences which have
        the intersection property are $\mathscr{L}^1$ and $\mathscr{L}^{12}$. Accordingly, by virtue of the homogeneity assumption, if $D^1=0$ and $\bX$ does not lie on the generator common to $Q$ and $Q_1$ then it must lie on the generator common to $Q_{\bar{1}}$ and $Q$. Thus, for a discrete Godeaux-Rozet surface, the set of generic envelopes (cf.\ Remark \ref{remgenenv}) consists of two one-parameter
        families. It turns out that these two families coincide in the following sense.
\end{remark}

\begin{theorem}\label{2=1}
   Let $\Omega$ with vertices labelled by $\bomega$ be an envelope associated with a set of lattice Lie quadrics $\{Q\}$ of a discrete Godeaux-Rozet surface for which the line congruence $\mathscr{L}$ has the intersection property in the $n_1$ direction
   and for which any vertex $\bomega$ does not lie on a
   generator shared by the lattice Lie quadrics $Q$ and $Q_1$.
   The discrete asymptotic net $\tilde{\Omega}$ defined by $\tilde{\bomega}= \bomega_1$ then constitutes another envelope for which $\tilde{\mathscr{L}}^1$ has the intersection property in $n_1$ direction and $\tilde{\mathscr{L}}^1\ni\tilde{l} = l_1\in\mathscr{L}$
    (cf.\ Figure~\ref{figdemcoin}).
\end{theorem}

\begin{proof}
By Remark \ref{remgr}, if $\bomega$ and $\bomega_1$ are two vertices
of the envelope $\Omega$ associated with quadrics $Q$ and $Q_1$
respectively, then
 $\bomega_1$ must lie on the shared generator of $Q$ and $Q_1$ since, by assumption, $\bomega$ does not. Hence, $\bomega_1$ may be regarded as a point $\tilde{\bomega}$ of the quadric $Q$.
 Since the quadrics $Q$ and $Q_1$ touch along the common generator, the tangent planes of $Q$ and $Q_1$ at $\tilde{\bomega}=\bomega_1$ coincide.
Hence, the envelope $\Omega$, having the intersection property with respect to $\mathscr{L}$, may also be interpreted as an envelope $\tilde{\Omega}$, having the intersection property with respect to the congruence $\tilde{\mathscr{L}}^1$
  consisting of the lines $\tilde{l} = [\tilde{\bomega},\mb{r}_1] = [\bomega_1,\mb{r}_1] = l_1\in\mathscr{L}$.
\end{proof}
    \begin{figure}[h]
        \centering
        \includegraphics[scale=1]{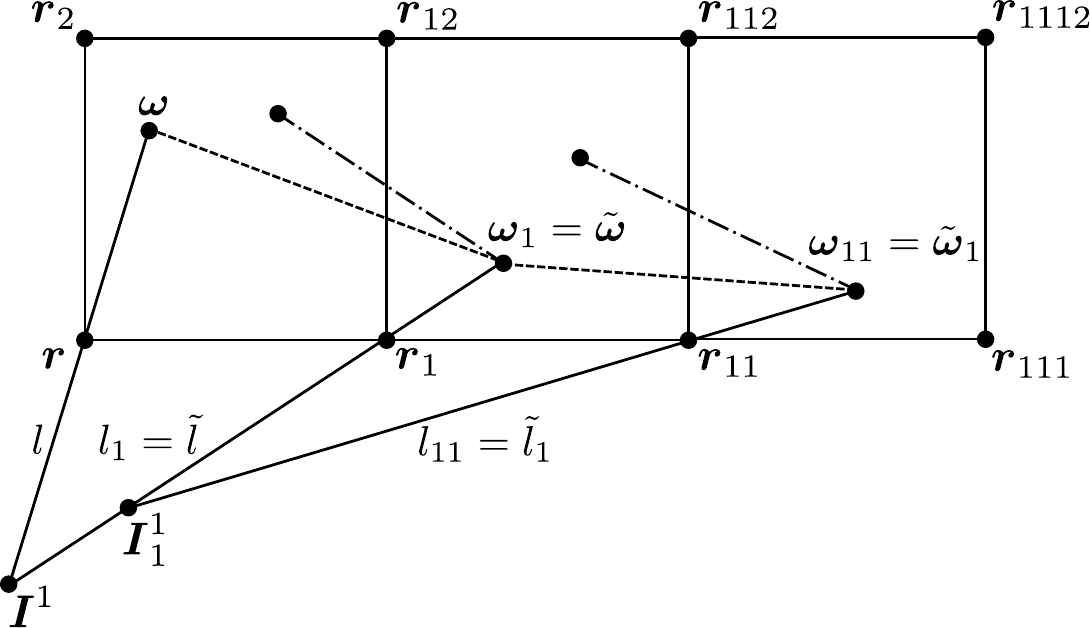}
        \caption{Coinciding envelopes of discrete Godeaux-Rozet surfaces -- shared generators are represented by dashed-dotted
        lines}
        \label{figdemcoin}
    \end{figure}
\begin{remark}\label{repx1q}
    We note that, for a discrete Godeaux-Rozet surface corresponding to $D_1=0$,
    if $\bomega$ and $\bomega_1$ are vertices of a generic envelope such that $\bomega=Q(s,t)$ does not lie on the
    generator common to $Q$ and $Q_1$ then $\bomega_1$
    as a point on $Q$ is represented by $\bomega_1=Q(\tilde{s},\tilde{t})$, where $\tilde{s}$
    is given by \eqref{eqsshare1} and, remarkably, $\tilde{t}=t$. Specifically, evaluation of $\bomega_1=Q_1(s_1,t_1)$ with $s_1$ and $t_1$ given by
\eqref{eqs2} and \eqref{eqt2} yields
\begin{equation*}
  \bomega_1 \sim p\mb{r}_{12}+\tilde{s}\mb{r}_{1}+t\mb{r}_{2}+\tilde{s}t\mb{r}
\end{equation*}
so that, indeed, $\bomega_1$ is the point of
    intersection of the generator common to $Q$ and $Q_1$ and the
    generator of $Q$ of the other type, passing through $\bomega$ and labelled by $t$.
\end{remark}
The above discussion shows that we can also characterise discrete
Demoulin surfaces in terms of line congruences. In particular, since
a discrete Demoulin surface is a Godeaux-Rozet surface with respect
to both the $n_1$ and $n_2$ directions, as a result of Theorem
\ref{thmgeogr1} and the fact that we may always choose $\bX\in Q$ such
that it does not lie on a generator common to $Q$ and $Q_1$ or $Q$
and $Q_2$, the following holds true.
    \begin{corollary}
        Under the assumptions of Theorem \ref{thmgeogr1}, a discrete projective minimal surface $\Sigma$ is a discrete Demoulin surface if and only if there exists a line congruence associated with any envelope which has the intersection property in both directions
        (as indicated in Figure \ref{figgeodem} for the line congruence~$\mathscr{L}$).
    \end{corollary}
    \begin{figure}[h]
        \centering
        \includegraphics[scale=1]{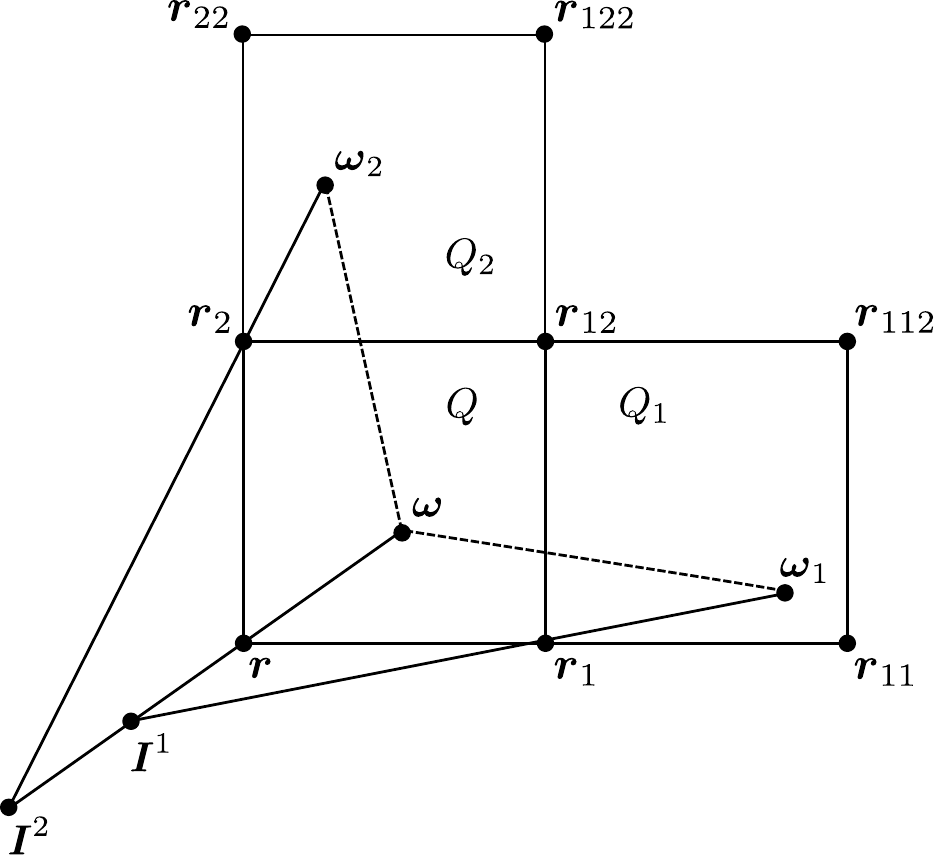}
        \caption{Intersection of three members of the line congruence $\mathscr{L}$ which has the intersection property in both directions.}
        \label{figgeodem}
    \end{figure}

    \begin{remark}
        Remark \ref{remgr} implies that for a discrete Demoulin
        surface and a given generic envelope $\Omega$, each
        vertex of $\Omega$ lies on a shared generator with one of the two
        neighbouring quadrics in the $n_1$ direction and on a shared generator with one of the two
        neighbouring quadrics in the $n_2$ direction. Thus, a
        discrete Demoulin surface has four generic envelopes, the four vertices of which associated with a quadric $Q$ constitute the points of intersection $\bA,\bB,\bC,\bD$ of the four generators of $Q$ shared with the neighbouring four quadrics as depicted in Figure \ref{figdem4}.
        By Remark \ref{repx1q}, the point $\bA_1$ is the intersection of the generator common to $Q$ and $Q_1$ and the generator common to $Q$ and $Q_{\bar{2}}$.
        Hence, it follows that
        $\bA_1=\bD$. Similarly, $\bB_1=\bC$ and for reasons of symmetry,
        $\bA_2=\bB$ and $\bD_2=\bC$. We conclude that the four envelopes of a discrete Demoulin surface coincide in the sense of Theorem~\ref{2=1}
        with $\bA_{12}=\bC$.
    \end{remark}
    \begin{figure}
        \centering
        \includegraphics[scale=0.915]{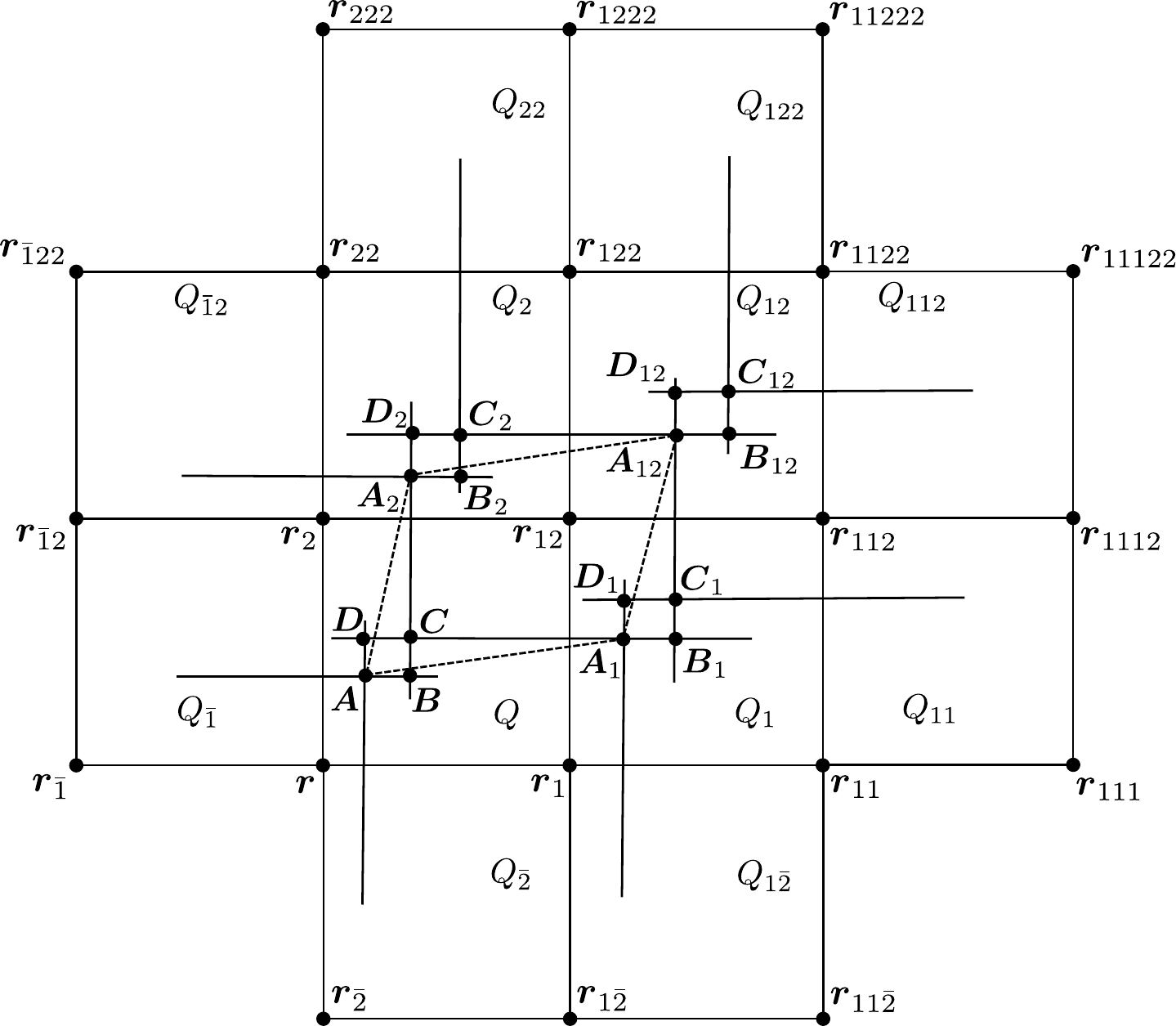}
        \caption{The vertices of the four envelopes of a discrete Demoulin
        surface regarded as the points of intersection of shared generators.}
        \label{figdem4}
    \end{figure}

In summary, the following statement may be made.

\begin{corollary} 
Discrete projective minimal surfaces and their subclasses may be characterised in terms of the number of envelopes of the associated set of lattice Lie quadrics.
\begin{enumerate}[label=(\roman*)]
  \item
  In general, a discrete projective minimal surface admits a two-parameter family of generic envelopes.
  \item
  In general, a discrete Godeaux-Rozet surface admits two one-parameter families of generic envelopes. The two families coincide modulo a relabelling of vertices.
   \item
   In general, a discrete Demoulin surface admits four generic envelopes which coincide modulo a relabelling of their vertices.
            \end{enumerate}
\end{corollary}

\subsection{Discrete Tzitz\'eica surfaces}

We have shown that discrete Demoulin surfaces admit four envelopes
which coincide modulo a relabelling of the vertices and that the
four associated line congruences which possess the intersection
property in both directions consist of the same lines. In the
following, we therefore focus without loss of generality on the
envelope of a discrete Demoulin surface for which $\mathscr{L}$
possesses the intersection property as depicted in Figure
\ref{figgeodem}. As in the preceding, we denote by $\bI^1$ the point
of intersection of a line $l\in\mathscr{L}$ and the neighbouring
line $l_1$ and, similarly, the point of intersection of the line $l$
and the neighbouring line $l_2$ is labelled by $\bI^2$. It is natural
to examine the case when the points $\bI^1$ and $\bI^2$ coincide. Since
$\bI^1$ and $\bI^2$ lie on $l$, this corresponds to one additional
condition on the discrete Gauss-Mainardi-Codazzi equations.
Comparison of the parametrisation \eqref{I1} of $\bI^1$ and its
counterpart for $\bI^2$ shows that this may be expressed as
$s_1t=t_2s$. Likewise, there exists only one condition for the
points $\bI^1$ and $\bI^1_1$ to coincide since both $\bI^1$ and $\bI^1_1$
lie on the line $l_1$. A brief calculation reveals that this
condition is equivalent to the condition of coinciding points $\bI^1$
and $\bI^2$. Hence, for reasons of symmetry, we have established the
following theorem.

    \begin{theorem}\label{thmtz}
        Let $\Sigma$ be a discrete Demoulin surface and $\bI^1$ and $\bI^2$  be the points of intersection of neighbouring lines of the line congruence $\mathscr{L}$. If the points $\bI^1$ and $\bI^2$ coincide for each quadrilateral then $\bI^1 = \bI^2$
        does not depend on the quadrilateral so that all lines of the line congruence $\mathscr{L}$ meet in a point.
    \end{theorem}

    In light of the above theorem, we define a special subclass of discrete Demoulin surfaces which may be regarded as discrete analogues of (projective transforms of) Tzitz\'eica surfaces (see, e.g., \cite{Schief2002} and references therein).
    \begin{definition}\label{deftz}
        A discrete projective minimal surface $\Sigma$ is said to be a projective transform of a discrete Tzitz\'eica surface if there exists an envelope and an associated line congruence such that all lines of the line congruence meet in a point.
    \end{definition}

Remarkably, the particular discrete projective
minimal surfaces $\Sigma$ defined above constitute projective
transforms of the integrable discrete Tzitz\'eica surfaces proposed
in \cite{bobenkoschief1999} in an entirely different context.
Indeed, we first note that the condition of coinciding points of
intersection $\bI^1$ and $\bI^2$, namely $s_1t=t_2s$, guarantees the
existence of a potential $\varphi$ related to $s$ and $t$ by
\begin{equation}\label{defphi}
   \varphi_1 = - t\varphi,\quad\varphi_2 = -s\varphi.
\end{equation}
This potential $\varphi$ turns out to be a scalar solution of the
discrete asymptotic net conditions (cf.\ \eqref{eqframe1})
\begin{equation}\label{asympdef}
   \mb{r}_{11} = \alpha^0\mb{r} + \alpha^1\mb{r}_1 +
\alpha^3\mb{r}_{12},\quad
   \mb{r}_{22} = \gamma^0\mb{r} + \gamma^2\mb{r}_2 + \gamma^3\mb{r}_{12}.
\end{equation}
For instance, substitution of $\varphi$ into \eqref{asympdef}$_1$
and evaluation by means of \eqref{defphi} yield
\begin{equation}
   t_1t = \alpha^0 - \alpha^1t + \alpha^3 s_1t,
\end{equation}
which is indeed satisfied by virtue of the expressions for $t_1$ and
$s_1$ obtained from \eqref{eqt2} and \eqref{eqtang4} respectively.
Accordingly, in the affine gauge
\begin{equation}
   \hat{\mb{r}}=\frac{\mb{r}}{\varphi},
\end{equation}
the discrete asymptotic net conditions adopt the form
\begin{equation}\label{asympaffine}
\begin{split}
   \hat{\mb{r}}_{11} - \hat{\mb{r}}_1 &=
\hat{\alpha}^1(\hat{\mb{r}}_1-\hat{\mb{r}}) +
\hat{\alpha}^3(\hat{\mb{r}}_{12}-\hat{\mb{r}}_1)\\
   \hat{\mb{r}}_{22} - \hat{\mb{r}}_2 &=
\hat{\gamma}^2(\hat{\mb{r}}_2-\hat{\mb{r}}) +
\hat{\gamma}^3(\hat{\mb{r}}_{12}-\hat{\mb{r}}_2).
\end{split}
\end{equation}
On the other hand, the constancy of the point of intersection
$I^1=I^2$ may be formulated as
\begin{equation}
   I^1\sim \mb{e}_4,\qquad \mb{e}_4 = (0\,\,0\,\,0\,\,1)^T
\end{equation}
modulo an appropriate projective transformation so that the
parametrisation \eqref{I1} gives rise to
\begin{equation}
   p\frac{\varphi_{12}\varphi}{\varphi_1\varphi_2}(\hat{\mb{r}}_{12} +
\hat{\mb{r}}) - (\hat{\mb{r}}_1 + \hat{\mb{r}}_2) \sim\mb{e}_4.
\end{equation}
If we now interpret the first three components $\mb{r}^{a}$ of
$\hat{\mb{r}}$ as the position vector of a representation $\Sigma^a$
in centro-affine geometry of the discrete surface $\Sigma$ then, by
virtue of \eqref{asympaffine}, $\Sigma^a$ constitutes a discrete
asymptotic net which is constrained by
\begin{equation}\label{affinenormal}
   \mb{r}^a_{12} + \mb{r}^a = h(\mb{r}^a_1 + \mb{r}^a_2),\qquad h =
\frac{\varphi_1\varphi_2}{p\varphi_{12}\varphi}.
\end{equation}
The latter may be interpreted geometrically if we define a discrete
affine normal $\mb{N}^a$ associated with a quadrilateral
$[\mb{r}^a,\mb{r}^a_1,\mb{r}^a_2,\mb{r}^a_{12}]$ to be the line
connecting the midpoints of the diagonals of the quadrilateral.
Thus, the constraint \eqref{affinenormal} shows that all affine
normals $\mb{N}^a$ meet at a point (namely the origin of the
coordinate system). This is precisely the property on which the
definition of the discrete affine spheres proposed in
\cite{bobenkoschief1999} is based.

\bigskip
\noindent
{\bf Funding.} This work was supported by the Australian Research Council (ARC Discovery Project DP140100851).’

\bibliographystyle{unsrt}
\bibliography{ref}
\end{document}